\documentclass{amsart}
\usepackage{amsmath}
\usepackage{amssymb}
\usepackage{amsthm}
\usepackage{graphicx}
\usepackage{color}
\usepackage{bm}

\newcommand{\widesim}[2][-1.5]{
   \mathrel{\overset{#2}{\scalebox{#1}[1]{$\sim$}}}
}

\newtheorem{theorem}{Theorem}[section]
\newtheorem{proposition}[theorem]{Proposition}

\newtheorem{lemma}[theorem]{Lemma}
\theoremstyle{definition}
\newtheorem{definition}[theorem]{Definition}
\newtheorem{example}[theorem]{Example}

\newtheorem{remark}[theorem]{Remark}

\newcommand{\A}{\mathcal{A}}

\begin{document}

\title[]{Palettes of Dehn colorings for spatial graphs and the classification of vertex conditions}


\author[K.~Oshiro]{Kanako Oshiro}
\address{Department of Information and Communication Sciences, Sophia University, Tokyo 102-8554, Japan}
\email{oshirok@sophia.ac.jp}

\author[N.~Oyamaguchi]{Natsumi Oyamaguchi}
\address{Department of Teacher Education, Shumei University, Chiba 276-0003, Japan}
\email{p-oyamaguchi@mailg.shumei-u.ac.jp}

\maketitle

\begin{abstract}
In this paper, we study Dehn colorings of spatial graph diagrams, and classify the vertex conditions, equivalently the {\it palettes}. We give some example of spatial graphs which can be distinguished by the number of Dehn colorings with selecting an appropriate palette. Furthermore, we also discuss the generalized version of palettes, which is defined for knot-theoretic ternary-quasigroups and region colorings of spatial graph diagrams.
\end{abstract}


\section{Introduction}
Throughout this paper, $\mathbb Z_+$ means the set of positive integers and $\mathbb Z_{\geq q}$ means the set of integers greater than or equal to some number $q$. Let $p\in \mathbb Z_{\geq 2}$, and let $\mathbb Z_p=\{0,1, \ldots, p-1\}$ denote the cyclic group $\mathbb Z/p\mathbb Z$.

Fox colorings, that is arc colorings by $\mathbb Z_p$ with some condition, were invented by Fox \cite{Fox} and have studied in knot theory, see \cite{HararyKauffman, Przytycki95, Satoh07} for example. 
The number of Fox $p$-coloring is a classical link invariant.  Ishii and Yasuhara applied it for spatial graphs and considered two kinds of coloring conditions for vertices in \cite{IshiiYasuhara97}.
After that, the first author introduced \textit{palettes}\footnote{We use spelling ``palette" in this paper, while the first author used spelling ``pallet" in \cite{Oshiro12}.} of Fox colorings, each of which gives a coloring condition for vertices of spatial graph diagrams, in \cite{Oshiro12}. She completely classified the vertex conditions for Fox colorings of spatial graph diagrams.

Dehn colorings, that is region colorings by $\mathbb Z_p$ with some condition, for classical link diagrams have been also studied in knot theory, see \cite{CarterSilverWilliams13, MadausNewmanRussell17, Nelson, Niebrzydowski0, Niebrzydowski1, Niebrzydowski2, Oshiro18} for example. 
In \cite{CarterSilverWilliams13}, some relation between Fox colorings and Dehn colorings was established. We studied Dehn colorings for spatial graph diagrams, and showed that some invariants $\tau_p, \varepsilon_p, \mu_p$, and $\mu_{p, \tau}$ for an equivalence relation on $\displaystyle 
\bigcup_{n \in 2\mathbb{Z}_{+}} \mathbb Z_p^n$ give a spatial graph invariant called a {\it vertex-weight invariant} in our previous paper \cite{OshiroOyamaguchi20}.

In this paper, we make use of these four invariants $\tau_p, \varepsilon_p, \mu_p$, and $\mu_{p, \tau}$ to classify the vertex conditions for Dehn colorings of spatial graph diagrams. We introduce palettes for Dehn colorings of spatial graph diagrams, where a palette gives a vertex condition, and where in this paper we call palettes of region colorings ``{\it $\mathcal{R}$-palettes}" and those of arc colorings ``{\it $\mathcal{A}$-palettes}".
Then, we give an example of spatial graphs which can be distinguished by selecting an appropriate $\mathcal{R}$-palette.
 Furthermore, we also discuss a generalization of Dehn colorings and $\mathcal{R}$-palettes, which is given for knot-theoretic ternary-quasigroups and region colorings of ``oriented" (or ``unoriented") spatial graph diagrams.

This paper is organized as follows: in Section 1, we introduce $\mathcal{R}$-palettes for Dehn $p$-colorings and an equivalence relation on $\bigcup_{n\in 2\mathbb Z_+} \mathbb Z^n_p$. In Section 2, some invariants under the equivalence relation on $\bigcup_{n\in 2\mathbb Z_+} \mathbb Z^n_p$ are reviewed. 
By using them, we state our main theorem, which gives the classification of vertex conditions for Dehn $p$-colorings. 
In Section 3, an $\mathcal{A}$-palette of Fox $p$-colorings is reviewed. We give the proof of the main theorem in Section 4. 
In Section 5, we define a Dehn $(p, P)$-coloring of an unoriented spatial graph diagram with a palette $P$ and show that the number of Dehn $(p, P)$-colorings is an invariant of unoriented spatial graphs. Moreover, we give some example of spatial graphs which can be distinguished by the number of Dehn $(p, P)$-colorings with selecting an appropriate palette. In Section 6 and Section 7, we show that the notion of $\mathcal{R}$-palettes for Dehn $p$-colorings can be extended for knot-theoretic ternary-quasigroups and region colorings of ``oriented" or ``unoriented" spatial graph diagrams in general.

\section{An $\mathcal{R}$-palette for Dehn colorings}

We first give a definition of an $\mathcal{R}$-palette for Dehn $p$-colorings.

\begin{definition}\label{def:DihedralRpalettes}
An {\it $\mathcal{R}$-palette} $P$ for  Dehn $p$-colorings is a subset of $\displaystyle 
\bigcup_{n \in 2\mathbb{Z}_{+}} \mathbb Z_p^n$ satisfying the following conditions:

\begin{itemize}
\item[(i)]  If $(a_1, a_2, \ldots, a_n) \in P$, then $(a_2, \ldots, a_n, a_1) \in P$.
\item[(ii)]  If $(a_1, a_2,\ldots, a_n) \in P$, then
$$\Big(a, a_2+(-1)^2(a_1-a), \ldots , a_i+(-1)^i(a_1-a), \ldots , a_{n}+(-1)^n(a_1-a)\Big) \in P$$ for any $a \in \mathbb Z_p$.
\item[(iii)]  If $(a_1, a_2,\ldots, a_n) \in P$, then
$$\Big(a, a_1-a_2+a, \ldots , a_1-a_i+a, \ldots ,a_1-a_n+a\Big) \in P$$ for any $a \in \mathbb Z_p$.

\item[(iv)]  If $(a_1, a_2, a_3, \ldots, a_n) \in P$ and $n > 2$, then
$$(a_1, -a_1+a_2+a_3, a_3, \ldots, a_n) \in P.$$ 
\end{itemize}
\end{definition}

\begin{definition}\label{def:Requiv}
Two elements $\displaystyle \boldsymbol{a}, \boldsymbol{b} \in \bigcup_{n\in 2\mathbb Z_+} \mathbb Z^n_p$ are {\it equivalent} ($\boldsymbol{a} \sim \boldsymbol{b}$) if $\boldsymbol{a}$ and $ \boldsymbol{b}$  are related by a finite sequence of the following transformations:
\begin{itemize}
\item[(Op1)] $(a_1,  \ldots , a_n) \longrightarrow (a_2, \ldots , a_n, a_1 )$,
\item[(Op2)] $(a_1,  \ldots , a_n) \longrightarrow (a, a_2+(-1)^2(a_1-a), \ldots , a_i +(-1)^i (a_1-a), \ldots ,  a_n +(-1)^n(a_1-a) )$ for $a\in \mathbb Z_p$,
\item[(Op3)] $(a_1,  \ldots , a_n) \longrightarrow (a, a_1-a_2 + a, \ldots ,  a_1-a_i + a, \ldots ,  a_1 -a_n+a )$ for $a\in \mathbb Z_p$,
\item[(Op4)]  $(a_1,  \ldots , a_n) \longrightarrow (a_1, -a_1 + a_2 + a_3, a_3 , \ldots , a_n )$.
\end{itemize}
\end{definition}

We note that the equivalence relation $\sim$ is closed in each $\mathbb Z_p^n$ as mentioned in Lemma~\ref{lem:dihedralpalette}.

\begin{remark}
The inverse of (Op1) (resp. (i) of Definition~\ref{def:DihedralRpalettes}) is (Op1)${}^{n-1}$ (resp. (i)${}^{n-1}$). The inverse of (Op2) (resp. (ii) of Definition~\ref{def:DihedralRpalettes}) for $c\in \mathbb Z_p$ is (Op2) (resp. (ii) of Definition~\ref{def:DihedralRpalettes}) for $a_1 \in \mathbb Z_p$. The inverse of (Op3) (resp. (iii) of Definition~\ref{def:DihedralRpalettes}) for $c\in \mathbb Z_p$ is (Op3) (resp. (iii) of Definition~\ref{def:DihedralRpalettes}) for $a_1 \in \mathbb Z_p$. The inverse of (Op4) (resp. (iv) of Definition~\ref{def:DihedralRpalettes}) is (Op4)${}^{p-1}$ (resp. (iv)${}^{p-1}$ of Definition~\ref{def:DihedralRpalettes}). 
\end{remark}

\begin{example}\label{universal}
For any $n \in 2\mathbb Z_+$, let $U_{p,n}=\mathbb Z^n_p $ and let $\displaystyle U_{p}=\bigcup_{n \in  2{\mathbb Z_{+}}}U_{p,n}$, both of which are $\mathcal{R}$-palettes for Dehn $p$-colorings. 
We call $\displaystyle U_{p,n}$ the {\it universal $\mathcal{R}$-palette of length $n$} and $\displaystyle U_{p}$ the {\it universal $\mathcal{R}$-palette}.
\end{example}

\begin{example}\label{alternating}
For any $n \in 2\mathbb Z_+$, let 
\[A_{p,n}=\left\{ (a_1,\ldots,a_n) \in U_{p,n} ~ {\Bigg |} ~
\begin{array}{l}
  a_i=a_1 \textrm{ if $i$ is odd}, \\
 a_i=a_2 \textrm{ if $i$ is even} \end{array}
  \right\}
  \]
and let $\displaystyle A_{p}=\bigcup_{n \in  2{\mathbb Z_{+}}}A_{p,n}$, both of which are $\mathcal{R}$-palettes for Dehn $p$-colorings. We call $\displaystyle A_{p,n}$ the {\it alternating $\mathcal{R}$-palette of length $n$} and $\displaystyle A_{p}$ the {\it alternating $\mathcal{R}$-palette}.
\end{example}

In this paper, we denote by $\bigsqcup$ the union of disjoint sets. Clearly we have the following lemma:
\begin{lemma}\label{lem:dihedralpalette}
\begin{itemize}
\item[(1)]  The equivalence relation $\sim$ is closed under each $U_{p,n}$, that is, it holds that  $\displaystyle U_p / \sim ~ = \bigsqcup_{n\in 2\mathbb Z_+} \big( U_{p,n}/ \sim \big)$.
\item[(2)] A subset $\displaystyle P \subset U_p$ is an $\mathcal{R}$-palette for Dehn $p$-colorings if and only if 
$\displaystyle P= \bigcup_{\lambda\in \Lambda} C_\lambda$ for some set $\displaystyle \{ C_\lambda \}_{\lambda \in \Lambda} \subset  U_p/ \sim $.
\end{itemize}
\end{lemma} 

\noindent Lemma~\ref{lem:dihedralpalette} implies that for classifying the $\mathcal{R}$-palettes, it suffices to classify the equivalence classes of $U_{p,n}$ with respect to the relation $\sim$. 
That is, it is important to know  $U_{p,n}/\sim $, which is given in Theorem~\ref{main}.

\section{Main theorem}
For any positive integers $a$ and $b$, $a|b$ means that $a$ is a divisor of $b$.

Put $\bm{a}=(a_1, \ldots ,a_n)$. Note that for the definition of $\tau_p, \varepsilon_p, \mu_p$, and $\mu_{p, \tau}$, we treat $a_i \in \mathbb Z_p=\{0,1, \ldots, p-1\}$ for $i \in \{1,\ldots, n\}$ as an integer.
We define $\displaystyle \tau_p: U_p \longrightarrow \mathbb{Z}$ by
\[
\tau_p \big( \bm{a} \big)=\max\left\{ k \in \{1,\ldots,p\} ~ {\Bigg |} ~
\begin{array}{l}
 \ k|p, \\
 a_1+a_2 \equiv a_2+a_3 \equiv \cdots \equiv a_n+a_1 \pmod{k}
 \end{array}
  \right\}.
  \]

Suppose that $p$ is an even integer. We define
$\displaystyle \varepsilon_p: U_p \longrightarrow \mathbb{Z} \cup \{\infty\}$ by
\[
\varepsilon_p \big( \bm{a}\big)=
\begin{cases}
0 & \mbox{if } a_1+a_2 \equiv \cdots \equiv a_n+a_1 \equiv 0  \pmod{2},\\
1 & \mbox{if } a_1+a_2 \equiv  \cdots \equiv a_n+a_1 \equiv 1  \pmod{2},\\
\infty & {\rm otherwise.}
\end{cases}
\]

\

We define
$\displaystyle \mu_p: U_p \longrightarrow \mathbb{Z}$ by
\[
\mu_p \big( \bm{a} \big)=E\big( (a_1+a_2, \ldots, a_n+a_1) \big)-O\big(  (a_1+a_2, \ldots ,a_n+a_1)\big),
\]
where 
\[E\big( \bm{a}\big)=\# \{ i \in \{1,\ldots, n\} \mid a_i \equiv 0 \pmod{2}\}\]
 and
\[O\big( \bm{a} \big)=\# \{ i \in \{1,\ldots, n\} \mid a_i \equiv 1 \pmod{2}\}.\]

For $\tau \in \{1, \ldots ,p\}$ such that $\tau \equiv 0 \pmod{2}$, $\tau|p$ and $\displaystyle \frac{p}{\tau} \equiv 0 \pmod{2}$, define $\displaystyle \mu_{p, \tau}: U_p \longrightarrow \mathbb{Z} \cup\{\infty\}$ by
\[
\mu_{p,\tau}\big( \bm{a} \big)
=
\begin{cases}\displaystyle
~ {\Bigg |} ~
\mu_{\frac{p}{\tau}} \Bigg( \Big( \frac{a_{1}-a_1}{\tau},\frac{a_{2}-a_2}{\tau}, \ldots , \frac{a_{2j-1}-a_1}{\tau},\frac{a_{2j}-a_2}{\tau}, \ldots , \frac{a_n-a_2}{\tau} \Big) \Bigg)
~ {\Bigg |}  \\
\hspace{85mm} \mbox{if } \tau_p(\bm{a})=\tau,\\
 \hspace{10mm} \infty  \hspace{75mm} {\rm otherwise.}
\end{cases}
\]
\begin{lemma}[\cite{OshiroOyamaguchi20}]
$\tau_p, \varepsilon_p, \mu_p$, and $\mu_{p, \tau}$ are invariants under the relation $\sim$.
\end{lemma}\label{2}

We then have our main theorem, which gives the complete classification of coloring conditions for vertices of spatial graph diagrams.
\begin{theorem}\label{main}
(1) Let $n=2$.

(i) When $p$ is an odd integer, we have 
\[
U_{p,2}/{\sim}=\left\{ U_{p,2}\right\}.
\]

(ii) When $p$ is an even integer, we have
\[
U_{p,2}/{\sim}=\left\{ \eta_{\varepsilon} \mid \varepsilon \in \{0, 1\} \right\},\]
where
\[
\eta_{\varepsilon}=\{\bm{a} \in U_{p,2} \mid \varepsilon_p(\bm{a})=\varepsilon \pmod{2}\}.
\]

(2) Let $n$ be an even integer greater than 2. 

(i) When p is an odd integer, we have
\[U_{p,n}/{\sim}=\left\{ \delta_{\tau}\mid \tau \in \{1,\ldots,p\} \mbox{ s.t. } \tau|p \right\},\]
where
\[
\delta_{\tau}=\{\bm{a} \in U_{p,n} \mid \tau_p(\bm{a})=\tau\}.
\]

(ii) When $p$ is an even integer, we have
\begin{eqnarray*}
&&U_{p,n}/{\sim}\\
&=&\left\{\alpha_{\tau, \mu}  {\Big |} 
\begin{array}{l}
 \tau \in \{1,\ldots,p\} \mbox{ s.t. } (\tau|p) \wedge \big(\tau \equiv 1 {\rm \  (mod \ 2)} \big), \\
\mu \in \mathbb{Z} \mbox{ s.t. } (-n<\mu<n) \wedge \big(\mu \equiv 0 {\rm \  (mod \ 2)}\big) \wedge \big(\frac{n-|\mu|}{2} \equiv 0 {\rm \  (mod \ 2)}\big)
\end{array}\hspace{-2mm}
\right\}\\
&\bigcup&  \left\{\beta_{\tau, \varepsilon}  {\Big |} 
\begin{array}{l}
 \tau \in \{1,\ldots,p\} \mbox{ s.t. } (\tau|p) \wedge \big(\tau \equiv 0 {\rm \  (mod \ 2)}\big) \wedge \big( \frac{p}{\tau} \equiv 1 {\rm \  (mod \ 2)}\big),\\
 \varepsilon \in \{0,1\}
  \end{array}\hspace{-2mm}
   \right\}\\
&\bigcup &\left\{\gamma_{\tau, \varepsilon, \mu}  {\Bigg |} 
\begin{array}{l}
 \tau \in \{1,\ldots,p\} \mbox{ s.t. } (\tau|p) \wedge \big(\tau \equiv 0 {\rm \  (mod \ 2)}\big) \wedge \big( \frac{p}{\tau} \equiv 0 {\rm \  (mod \ 2)}\big),\\
 \varepsilon \in \{0,1\},\\
 \mu \in \mathbb{Z} \mbox{ s.t. } (0 \leq \mu<n) \wedge \big(\mu \equiv 0 {\rm \  (mod \ 2)}\big) \wedge \big(\frac{n-|\mu|}{2} \equiv 0 {\rm \  (mod \ 2)}\big)
  \end{array}\hspace{-2mm}
   \right\},
   \end{eqnarray*}
where 
\[
\alpha_{\tau, \mu}=\{\bm{a} \in U_{p,n} \mid \tau_p(\bm{a})=\tau, ~\mu_p(\bm{a})=\mu \},
\]
\[
\beta_{\tau, \varepsilon}=\{\bm{a} \in U_{p,n} \mid \tau_p(\bm{a})=\tau, ~\varepsilon_p(\bm{a})=\varepsilon \},
\]
and
\[
\gamma_{\tau, \varepsilon, \mu}=\{\bm{a} \in U_{p,n} \mid \tau_p(\bm{a})=\tau, ~\varepsilon_p(\bm{a})=\varepsilon, 
~\mu_{p, \tau}(\bm{a})=\mu \}.
\]
\end{theorem}

\begin{table}[ht]
\begin{scriptsize}
  {\begin{tabular}{|c|c|c|c|c|c|} \hline
   $U_{p,n}^{}$ & $U_{p,n}/{\sim}$ & $\tau_p$ & $\varepsilon_p$ & $\mu_p$ & $\mu_{p,\tau}$\\ \hline
\begin{minipage}{1.1cm}       
$n=2$,\\
$p$: odd
\end{minipage}
 & $U_{p,2}$ & $\langle\tau=p\rangle$ & - & -& -\\ \hline  
$n=2$,
 & $\eta_{\varepsilon}$  & $\langle\tau=p\rangle$ & $\varepsilon=0$  & $\langle \mu=2\rangle$& -\\ \cline{4-5} 
$p$: even &  &  & $\varepsilon=1$  & $\langle \mu=-2\rangle$& \\ \hline
\begin{minipage}{1.1cm}       
$n \geq 4$,\\
$p$: odd
\end{minipage} & $\delta_{\tau}$ & \begin{minipage}{1.8cm}   
$\tau\in \{1,\ldots ,p \}\\
\mbox{s.t. } \tau|p$
\end{minipage} &- & - & - \\  \hline  
     & $\alpha_{\tau, \mu}$ & 
\begin{minipage}{1.8cm}   
$\tau\in \{1,\ldots ,p \}\\
\mbox{s.t. } (\tau|p)\\
 \wedge (\tau \equiv_{2} 1)$
\end{minipage} & $\langle\varepsilon=\infty\rangle$&\begin{minipage}{2.6cm}  $-n<\mu<n$ \mbox{s.t. }\\   ($\mu \equiv_{2} 0 ) \wedge\\
 \big(\frac{n-|\mu|}{2} \equiv_{2} 0\big)$ \end{minipage} & -\\ \cline{2-6}
 \begin{minipage}[]{1.1cm}       
$n \geq 4$,\\
$p$: even
\end{minipage} 
& $\beta_{\tau, \varepsilon}$  & 
\begin{minipage}{1.8cm}   
$\tau\in \{1,\ldots ,p \}\\
 \mbox{s.t. } (\tau|p) \\
\wedge (\tau \equiv_{2} 0 )$
\end{minipage}
&  $\varepsilon=0$ & $\langle \mu=n\rangle$& -\\ \cline{4-6}
&  &\begin{minipage}{1.8cm} \begin{center}  
$ \wedge (\frac{p}{\tau} \equiv_{2} 1)$
\end{center}\end{minipage} &  $\varepsilon=1$ & $\langle\mu=-n\rangle$& -\\ \cline{2-6}
& $\gamma_{\tau, \varepsilon, \mu}$ &\begin{minipage}{1.8cm}   
$\tau\in \{1,\ldots ,p \}
$
\end{minipage}
&  $\varepsilon=0$ & $\langle \mu=n\rangle$& $0 \leq \mu<n$\\ \cline{4-5}
      & & 
\begin{minipage}{1.8cm}   
$
 \mbox{s.t. } (\tau|p) \\
 \wedge (\tau \equiv_{2} 0)\\
 \wedge (\frac{p}{\tau} \equiv_{2} 0)$
\end{minipage}
& $\varepsilon=1$ & $\langle \mu=-n\rangle$ & 
\begin{minipage}{2cm}   \mbox{s.t. } $(\mu \equiv_{2} 0)$\\ $\wedge (\frac{n-\mid\mu\mid}{2}\equiv_{2} 0)$\end{minipage}\\ \hline
\end{tabular}}
\label{Rpalettetable}
  \end{scriptsize}
    \end{table}

Table 1 shows the values by the maps $\tau_p, \varepsilon_p, \mu_p$, and $\mu_{p, \tau}$ for elements of each equivalence class. Note that $a \equiv_{2} b$ means that $a$ is congruent to $b$ modulo 2. We also note that the values with angle brackets $\langle \ \rangle$ are not used for classifying the equivalence classes. 


\section{An $\A$-palette for Fox colorings}

The first author \cite{Oshiro12} introduced the palettes of arc colorings for spatial graph diagrams in 2012. 
For any $n \in 2\mathbb Z_+$, let 
\[
U^{\A}_{p,n}=\{ (x_1, x_2, \ldots ,x_n)^{\A} \in \mathbb Z_p^n \mid 2\sum_{i=1}^n(-1)^i x_i \equiv 0 \pmod{p}\}
\]
and let $\displaystyle U_{p}^{\A}=\bigcup_{n \in  2{\mathbb Z_{+}}}U_{p,n}^{\A}$,
 where we put superscript $\A$ to the elements of $\displaystyle U_{p,n}^{\A}$ to distinguish them from the elements of $\displaystyle U_{p,n}$.

\begin{definition}\label{def:Aequiv}
Two elements $\displaystyle \boldsymbol{x}^{\A}, \boldsymbol{y}^{\A} \in U_p^{\A}$ are {\it equivalent} ($\boldsymbol{x}^{\A} \sim^{\A} \boldsymbol{y}^{\A}$)  if $\boldsymbol{x}^{\A}$ and $\boldsymbol{y}^{\A}$  are related by a finite sequence of the following transformations:

\begin{itemize}
\item[(Op1)$^{\A}$] $(x_1,  \ldots , x_n)^{\A} \longrightarrow (x_2, \ldots , x_n, x_1 )^{\A}$,
\item[(Op2)$^{\A}$] $(x_1,  \ldots , x_n)^{\A} \longrightarrow (2x-x_1,  \ldots , 2x-x_n)^{\A}$ for $x\in \mathbb{Z}_p$,
\item[(Op3)$^{\A}$] $(x_1,  \ldots , x_n)^{\A} \longrightarrow (x_2, 2x_2-x_1, x_3 , \ldots , x_n)^{\A}$.
\end{itemize}

\end{definition}

An {\it $\A$-palette   $P^{\A}$} of Fox $p$-colorings is a subset of $\displaystyle U_p^{\A}$ such that the equivalence class of any element of $P^{\A}$ is a subset of $P^{\A}$.

Put $\bm{x}^\A=(x_1, \ldots ,x_n)^\A$. Note that for the definitions of $\tau_p^{\A}, \varepsilon_p^{\A}, \mu_p^{\A}$, and $\mu_{p, \tau}^{\A}$, we treat $x_i \in \mathbb Z_p=\{0,1, \ldots, p-1\}$ for $i \in \{1,\ldots, n\}$ as an integer. We define
$\displaystyle \kappa^{\A}_{p}: U_p^{\A} \longrightarrow \mathbb{Z}_p$ by
\[
\kappa^{\A}_{p} \big( \bm{x}^{\A} \big)=\sum^n_{i=1}(-1)^i x_i,
\]and 
$\displaystyle \tau^{\A}_{p}: U_p^{\A} \longrightarrow \mathbb{Z}$ by
\[
\tau^{\A}_{p} \big( \bm{x}^{\A} \big)=\textbf{max}\{k \in \{1,\ldots,p\} \mid k|p, \ x_1 \equiv x_2 \equiv \cdots \equiv x_n \pmod{k}\}.
\]

Suppose that $p$ is an even integer. We define
$\displaystyle \varepsilon^{\A}_{p}: U_p^{\A} \longrightarrow \mathbb{Z}\cup \{\infty\}$ by
\[
\varepsilon^{\A}_{p} \big( \bm{x}^{\A}\big)=
\begin{cases}
0 & \mbox{if } x_1 \equiv x_2 \equiv \cdots \equiv x_n \equiv 0  \pmod{2},\\
1 & \mbox{if } x_1 \equiv x_2 \equiv \cdots \equiv x_n \equiv 1 \pmod{2},\\
\infty & \mbox{otherwise}.\end{cases}
\]
and
$\displaystyle \mu^{\A}_{p}: U_p^{\A} \longrightarrow \mathbb{Z}$ by
\[
\mu^{\A}_{p} \big( \bm{x}^{\A} \big)=E^{\A}\big(\bm{x}^{\A}\big)-O^{\A}\big(\bm{x}^{\A}\big),
\]
where 
\[E^{\A}\big(\bm{x}^{\A}\big)=\# \{ i \in \{1,\ldots, n\} \mid x_i \equiv 0 \pmod{2}\}\]
 and
\[O^{\A}\big(\bm{x}^{\A}\big)=\# \{ i \in \{1,\ldots, n\} \mid x_i \equiv 1 \pmod{2}\}.\]

Let $\tau \in \{1,\ldots,p\}$ be an even divisor of $p$ such that $\displaystyle \frac{p}{\tau}$ is an even integer. We define
$\displaystyle \mu^{\A}_{p, \tau}: U_p^{\A} \longrightarrow \mathbb{Z}\cup\{\infty\}$ by

\[
\mu^{\A}_{p, \tau}\big( \bm{x}^{\A} \big)
=
\begin{cases}\displaystyle
~ {\Bigg |} ~
\mu^{\A}_{{\frac{p}{\tau}}} \Bigg( \Big( 0,\frac{x_{2}-x_1}{\tau}, \ldots , \frac{x_n-x_1}{\tau} \Big)^{\A} \Bigg)
~ {\Bigg |}  & \mbox{if } \tau_p^{\A}(\bm{x}^{\A})=\tau,\\
\hspace{30mm} \infty & {\rm otherwise.}
\end{cases}
\]

\noindent We have the following theorem.
\begin{theorem}[\cite{Oshiro12}]\label{Amain}
(1) Let $n=2$. 

(i) When $p$ is an odd integer, we have 
\[
U^{\A}_{p,2}/{\sim}^{\A}=\left\{ U_{p,2}^{\A}\right\}.
\]

(ii) When $p$ is an even integer such that $\displaystyle \frac{p}{2}$ is an odd integer, we have
\[
U_{p,2}^{\A}/{\sim}^{\A}=\left\{ \eta^{\A}_{\mu} \mid \mu \in \{-2, 0, 2 \}\right\},\]
where
\[
\eta^{\A}_{\mu}=\{\bm{x}^{\A} \in U_{p,2}^{\A} \mid \mu^{\A}_{p} \big( \bm{x}^{\A} \big)=\mu \}.
\]

(iii) When $p$ is an even integer such that $\displaystyle \frac{p}{2}$ is an even integer, we have
\[
U_{p,2}^{\A}/{\sim}^{\A}=\left\{ \eta^{\A}_{\tau, \varepsilon} \mid \tau \in \left\{\frac{p}{2}, p\right\}; \varepsilon \in \{0, 1\}\right\},\]
where
\[
\eta^{\A}_{\tau, \varepsilon}=\{\bm{x}^{\A} \in U_{p,2}^{\A} \mid \tau^{\A}_{p} \big( \bm{x}^{\A} \big)=\tau, \varepsilon^{\A}_p(\bm{x}^{\A})=\varepsilon \}.
\]
   
(2) Let $n$ be an even integer greater than 2.

 (i) When p is an odd integer, we have
\[U_{p,n}^{\A}/{\sim}^{\A}=\left\{ \delta^{\A}_{\tau}\mid \tau \in \{1,\ldots,p\} \mbox{ s.t. } \tau|p \right\},\]
where
\[
\delta^{\A}_{\tau}=\{\bm{x}^{\A} \in U^{\A}_{p,n} \mid \tau^{\A}_{p}(\bm{x}^{\A})=\tau\}.
\]

(ii) When $p$ is an even integer, we have
\begin{eqnarray*}
&&U^{\A}_{p,n}/{\sim}^{\A}\\
&=&\left\{\alpha^{\A}_{\kappa, \tau, \mu}  {\Big |} 
\begin{array}{l}
\kappa \in \{0, \frac{p}{2}\}, \tau \in \{1,\ldots,p\} \mbox{ s.t. } (\tau|p) \wedge \big(\tau \equiv 1 {\rm \  (mod \ 2)}\big),\\
\mu \in \mathbb{Z} \mbox{ s.t. } (-n<\mu<n) \wedge \big(\mu \equiv 0 {\rm \  (mod \ 2)}\big) \wedge \big(\frac{n-|\mu|}{2} \equiv \kappa {\rm \  (mod \ 2)}\big)
\end{array}\hspace{-2mm}
\right\}\\
&\bigcup&  \left\{\beta^{\A}_{\tau, \varepsilon}  {\Big |} 
\begin{array}{l}
 \tau \in \{1,\ldots,p\} \mbox{ s.t. } (\tau|p) \wedge \big(\tau \equiv 0 {\rm \  (mod \ 2)}\big) \wedge \big( \frac{p}{\tau} \equiv 1 {\rm \  (mod \ 2)}\big),\\
 \varepsilon \in \{0,1\}
  \end{array}\hspace{-2mm}
   \right\}\\
&\bigcup& \left\{\gamma^{\A}_{\kappa, \tau, \varepsilon, \mu}  {\Bigg |} 
\begin{array}{l}
\kappa \in \{0, \frac{p}{2}\},  \varepsilon \in \{0,1\},\\
\tau \in \{1,\ldots,p\} \mbox{ s.t. } (\tau|p) \wedge \big(\tau \equiv 0 {\rm \  (mod \ 2)}\big) \wedge \big( \frac{p}{\tau} \equiv 0 {\rm \  (mod \ 2)}\big), \\
   \mu \in \mathbb{Z} \mbox{ s.t. } (0 \leq \mu<n) \wedge \big(\mu \equiv 0 {\rm \  (mod \ 2)}\big) \wedge \big(\frac{n-\mu}{2} \equiv \frac{\kappa}{\tau} {\rm \  (mod \ 2)}\big)
  \end{array}\hspace{-2mm}
   \right\},
   \end{eqnarray*}
where 
\[
\alpha^{\A}_{ \kappa, \tau, \mu}=\{\bm{x}^{\A} \in U_{p,n}^{\A} \mid ~\kappa^{\A}_p(\bm{x}^{\A})=\kappa,  \tau^{\A}_p(\bm{x}^{\A})=\tau, ~\mu_p^{\A}(\bm{x}^{\A})=\mu \},
\]
\[
\beta^{\A}_{\tau, \varepsilon}=\{\bm{x}^{\A} \in U^{\A}_{p,n} \mid \tau^{\A}_p(\bm{x}^{\A})=\tau, ~\varepsilon_p^{\A}(\bm{x}^{\A})=\varepsilon \},
\]
and
\[
\gamma^{\A}_{\kappa, \tau, \varepsilon, \mu}=\{\bm{x}^{\A} \in U^{\A}_{p,n} \mid ~\kappa^{\A}_p(\bm{x}^{\A})=\kappa, \tau^{\A}_p(\bm{x}^{\A})=\tau, ~\varepsilon^{\A}_p(\bm{x}^{\A})=\varepsilon, ~\mu^{\A}_{p, \tau}(\bm{x}^{\A})=\mu \}.
\]
\end{theorem}

\begin{table}[ht]
\begin{scriptsize}
{  \begin{tabular}{|c|c|c|c|c|c|c|} \hline
   $U_{p,n}^{\A}$ & $U_{p,n}^{\A}/{\sim}^{\A}$ & $\kappa_p^{\A}$&$\tau_p^{\A}$ & $\varepsilon_p^{\A}$ & $\mu_p^{\A}$ & $\mu_{p,\tau}^{\A}$\\ \hline
\begin{minipage}{1cm}       
$n=2$,\\
$p$: odd
\end{minipage}
 & $U_{p,2}^{\A}$ & $\langle\kappa=0\rangle$&$\langle\tau=p\rangle$  & -& -&-\\ \hline
   \begin{minipage}{1cm}       
$n=2$,
\end{minipage} & $\eta_{\mu}^{\A}$  & $\langle\kappa=0\rangle$& $\langle\tau=p\rangle$ & $\langle\varepsilon=0\rangle$  & $\mu=2$& -\\ \cline{5-7} 
$p$: even, &  & & & $\langle\varepsilon=1\rangle$  & $\mu=-2$& -\\ \cline{3-7}
 $\frac{p}{2}$: odd &  & $\langle\kappa=\frac{p}{2}\rangle$& $\langle\tau=\frac{p}{2}\rangle$ &  $\langle\varepsilon=\infty\rangle$  & $\mu=0$& -\\ \hline
$n=2$, &  $\eta_{\tau, \varepsilon}^{\A}$  & $\langle\kappa=0\rangle$ & $\tau=p$ &$\varepsilon=0$  & $\langle\mu=2\rangle$& -\\ \cline{5-7} 
$p$: even,  &  &  &  &$\varepsilon=1$  & $\langle\mu=-2\rangle$& -\\ \cline{3-7} 
       $\frac{p}{2}$: even
 & & $\langle\kappa=\frac{p}{2}\rangle$ & $\tau=\frac{p}{2}$& $\varepsilon=0$  & $\langle\mu=2\rangle$& -\\ \cline{5-7} 
 &  &  &  &$\varepsilon=1$  & $\langle\mu=-2\rangle$ & \\ \hline
\begin{minipage}{1cm}       
$n \geq 4$,\\
$p$: odd
\end{minipage} & $\delta_{\tau}^{\A}$ &$\langle\kappa=0\rangle$& \begin{minipage}{1.8cm}   
$\tau\in \{1,\ldots ,p \}\\
\mbox{s.t. } \tau|p$
\end{minipage} &- & - & - \\  \hline  
     & $\alpha_{\kappa, \tau, \mu}^{\A}$ & $\kappa=0$&
\begin{minipage}{1.8cm}   
$\tau\in \{1,\ldots ,p \}\\
\mbox{s.t. } (\tau|p) \wedge$
\end{minipage} & $\langle\varepsilon=\infty\rangle$&\begin{minipage}{1.8cm}   $-n<\mu<n$ \\ \mbox{s.t. }
\end{minipage} & -\\ \cline{3-3}
& & $\kappa=\frac{p}{2}$&\begin{minipage}{1.8cm}$ (\tau \equiv_{2} 1)$\end{minipage} & & \begin{minipage}{1.8cm}  $(\mu \equiv_{2} 0 ) \wedge$\\$(\frac{n-\mid \mu \mid}{2}\equiv_{2} \kappa)$
\end{minipage} & \\ \cline{2-7}
 \begin{minipage}{1.2cm}       
$n \geq 4$,\\
$p$: even
\end{minipage} 
& $\beta_{\tau, \varepsilon}^{\A}$  &$\langle \kappa=0\rangle$& 
\begin{minipage}[l]{1.8cm}   
$\tau\in \{1,\ldots ,p \}\\
 \mbox{s.t. } (\tau|p) \wedge $
\end{minipage}
&  $\varepsilon=0$ & $\langle\mu=n\rangle$& -\\ \cline{5-7}
& & &\begin{minipage}[l]{1.8cm}   $(\tau \equiv_{2} 0) \wedge\\ (\frac{p}{\tau} \equiv_{2} 1)$
\end{minipage}
 &   $\varepsilon=1$ & $\langle\mu=-n\rangle$& -\\ \cline{2-7}
& $\gamma_{\kappa, \tau, \varepsilon, \mu}^{\A}$ &$\kappa=0$&\begin{minipage}{1.8cm}   
$\tau\in \{1,\ldots ,p \}$
\end{minipage}
&  $\varepsilon=0$ & $\langle\mu=n\rangle$& 
\begin{minipage}{1.5cm} $0 \leq \mu<n$ \end{minipage}
\\ \cline{5-6}
      & & &\begin{minipage}{1.5cm}  \mbox{s.t. } $(\tau|p) \wedge$\end{minipage}& $\varepsilon=1$ & $\langle\mu=-n\rangle$ & \begin{minipage}{1.5cm}  \mbox{s.t. }  \end{minipage}
\\ \cline{5-6}\cline{3-3}
& & $\kappa=\frac{p}{2}$ &\begin{minipage}{1.5cm}$(\tau \equiv_{2} 0) \wedge$\end{minipage}& $\varepsilon=0$ & $\langle\mu=n\rangle$ & \begin{minipage}{1.5cm}$(\mu \equiv_{2} 0)\wedge$ \end{minipage}
\\ \cline{5-6}
&  &&\begin{minipage}{1.5cm}$(\frac{p}{\tau} \equiv_{2} 0)$\end{minipage}&  $\varepsilon=1$ & $\langle\mu=-n\rangle$& \begin{minipage}{1.65cm}$(\frac{n-\mu}{2}\equiv_{2} \frac{\kappa}{\tau} )$\end{minipage}\\ 
\hline
  \end{tabular}}
    \end{scriptsize}
  \label{Apalettetable}
\end{table}

Table 2 shows the values by the maps $\kappa^{\A}_p, \tau^{\A}_p, \varepsilon^{\A}_p, \mu^{\A}_p$, and $\mu^{\A}_{p, \tau}$ for elements of each equivalence class.   Note that $a \equiv_{2} b$ means that $a$ is congruent to $b$ modulo 2. We also note that the values with angle brackets $\langle \ \rangle$ are not used for classifying equivalence classes. 

For example, the value $\kappa_{p}^{\A}(\bm{x}^{\A})=0$ for $\bm{x}^{\A}=(x_1, \ldots, x_n)^{\A} \in \beta_{\tau,\varepsilon}^{\A}$  is obtained as follows: Since $\tau_{p}^{\A}(\bm{x}^{\A})=\tau$, 

\begin{align*}
\bm{x}^{\A}&=(x_1, \ldots, x_n)^{\A}\\
&=(x_1, x_1+y_2\tau, x_1+y_3\tau, \ldots, x_1+y_n\tau)^{\A}
\end{align*}
for some $y_2, \ldots , y_n \in \{0, \ldots, p-1\}$. We then have 
\begin{align*}
\kappa^{\A}_p(\bm{x}^{\A})&=\kappa^{\A}_p\big((x_1, \ldots, x_n)^{\A}\big)\\
&=\kappa^{\A}_p\big((x_1, x_1+y_2\tau, x_1+y_3\tau, \ldots, x_1+y_n\tau)^{\A}\big)\\
&=-x_1+(x_1+y_2\tau)-(x_1+y_3\tau)+ \cdots +(x_1+y_n\tau)\\
&=y_2\tau-y_3\tau+ \cdots +y_n\tau\\
&=(y_2-y_3+ \cdots +y_n)\tau.
\end{align*}
Assume that $\kappa_{p}^{\A}(\bm{x}^{\A})=\frac{p}{2}$, where $\kappa_{p}^{\A}(\bm{x}^{\A}) \in \{0, \frac{p}{2}\}$ since $\bm{x}^{\A} \in U_p^{\A}$. We then have $ (y_2-y_3+ \cdots +y_n)\tau=\frac{p}{2}$. Hence $ 2(y_2-y_3+ \cdots +y_n)=\frac{p}{\tau}$, which contradicts the condition that $ \frac{p}{\tau}$ is odd. Therefore, we have $\kappa_{p}^{\A}(\bm{x}^{\A})=0$.
We leave the proof of the other values in the table to the reader.

We have the following properties for the equivalence classes such that $\kappa_p$ of their elements are 0.

\begin{lemma}[\cite{Oshiro12}]\label{Alemma}
(1) 
When $p$ is an odd integer, for $\bm{x}^{\A} \in U_{p,2}^{\A}$, it holds 
\[
\bm{x}^{\A}  \sim^{\A} (0,0)^{\A}.
\]

(2) 
For $\bm{x}^{\A} \in \eta^{\A}_{\mu}$, it holds 
\[
\bm{x}^{\A}  \sim^{\A} \begin{cases}
(0,0)^{\A} & {\rm if \ } \mu=2,\\
(1,1)^{\A} & {\rm if \ } \mu=-2.
\end{cases}
\]

(3) For $\bm{x}^{\A} \in \eta^{\A}_{\tau=p,\varepsilon}$, it holds 
\[
\bm{x}^{\A}  \sim^{\A} \begin{cases}
(0,0)^{\A} & {\rm if \ } \varepsilon=0,\\
(1,1)^{\A} & {\rm if \ } \varepsilon=1.
\end{cases}
\]

(4) 
For  $\bm{x}^{\A} \in \delta^{\A}_{\tau}$, it holds 
\[
\bm{x}^{\A}  \sim^{\A} (0, \tau, 0, \tau, \ldots, 0, \frac{2-n}{2}\tau)^{\A}.
\]

(5) 
For  $\bm{x}^{\A} \in \alpha^{\A}_{ \kappa=0, \tau, \mu}$, it holds 
\[
\bm{x}^{\A}  \sim^{\A} \begin{cases}
(0,\tau,0,\tau,\ldots,0,\frac{2-n-\mu}{2}\tau,\underbrace{\tau,\ldots,\tau}_{-\mu})^{\A} & {\rm if \ } -n<\mu\leq0,\\
(0,\tau,0,\tau,\ldots,0,\frac{2-n+\mu}{2}\tau,\underbrace{0,\ldots,0}_{\mu})^{\A}& {\rm if \ } 0<\mu<n.\end{cases}
\]
 
(6) 
For  $\bm{x}^{\A} \in \beta^{\A}_{\tau, \varepsilon}$, it holds 
\[
\bm{x}^{\A}  \sim^{\A} \begin{cases}
(0,\tau,0,\tau,\ldots,0,\frac{2-n}{2}\tau)^{\A} & {\rm if \ } \varepsilon=0,\\
(1,\tau+1,1,\tau+1,\ldots,1,\frac{2-n}{2}\tau+1)^{\A}& {\rm if \ } \varepsilon=1.\end{cases}
\]

(7) 
For  $\bm{x}^{\A} \in \gamma^{\A}_{\kappa=0, \tau, \varepsilon, \mu}$, it holds 
\[
\bm{x}^{\A}  \sim^{\A} \begin{cases}
(0,\tau,0,\tau,\ldots,0,\frac{2-n+\mu}{2}\tau,\underbrace{0,\ldots,0}_{\mu})^{\A} & {\rm if \ } \varepsilon=0,\\
(1,\tau+1,1,\tau+1,\ldots,1,\frac{2-n+\mu}{2}\tau+1,\underbrace{1,\dots,1}_{\mu})^{\A}& {\rm if \ } \varepsilon=1.\end{cases}
\]

\end{lemma}

\begin{remark}
In \cite{Oshiro12}, the properties shown in this section were discussed only when $p \geq 3$. However, we can see that these properties also hold for $p=2$.
\end{remark}

 \section{Proof of the main theorem}
 
Let $n\in 2 \mathbb Z_+$. Set 
\[
V_{p,n}^{\mathcal{A}}=\Big\{ (x_1, \ldots , x_n)^{\mathcal{A}} \in  U_p ~\Big|~ \sum_{i=1}^n (-1)^ix_i \equiv 0 \pmod{p} \Big\},
\] 
which is an $\mathcal{A}$-palette of Fox $p$-colorings.
Define a map 
$T_{\mathcal{R}\to \mathcal{A}}: \{0\} \times \mathbb Z_p^{n-1} \to  V_{p,n}^{\mathcal{A}} $
by 
\begin{align*}
T_{\mathcal{R}\to \mathcal{A}} \big( (0, a_2, \ldots ,  a_n) \big) 
&=(a_2, a_2 + a_3, \ldots ,a_{n-1}+a_n ,a_n )^{\mathcal{A}},
\end{align*}	
where $T_{\mathcal{R}\to \mathcal{A}}$ is regarded as a translation from an $n$-tuple of region colors to that of arc colors.
 Define a map $T_{\mathcal{A}\to \mathcal{R}}:   V_{p,n}^{\mathcal{A}}  \to  \{0\} \times \mathbb Z_p^{n-1}$ by 
 \begin{align*}
&T_{\mathcal{A}\to \mathcal{R}}\big( (x_1, x_2, \ldots , x_n)^{\mathcal{A}} \big) \\
&= \Big(0, \sum_{i=1}^{1} (-1)^{i+1} x_i , \sum_{i=1}^{2} (-1)^{i+2} x_i , \ldots , \sum_{i=1}^{n-1} (-1)^{i+(n-1)} x_i \Big),
\end{align*}
where $T_{\mathcal{A}\to \mathcal{R}}$ is regarded as a translation from an $n$-tuple of arc colors to that of region colors.

 \begin{lemma}
$ T_{\mathcal{R}\to \mathcal{A}} $ and $T_{\mathcal{A}\to \mathcal{R}}$ are  inverses of each other, and hence, both of $ T_{\mathcal{R}\to \mathcal{A}} $ and $T_{\mathcal{A}\to \mathcal{R}}$ are bijective.
 \end{lemma}
 \begin{proof}
The equality $T_{\mathcal{A}\to \mathcal{R}} \circ T_{\mathcal{R}\to \mathcal{A}}= {\rm id}$ follows from 
\begin{align*}
&T_{\mathcal{A}\to \mathcal{R}} \circ T_{\mathcal{R}\to \mathcal{A}} \big( (a_1=0, a_2, \ldots ,  a_n) \big)\\
&= T_{\mathcal{A}\to \mathcal{R}} \big( (a_1+a_2, a_2 + a_3, \ldots ,a_{n-1}+a_n ,a_n + a_1)^{\mathcal{A}}\big) \\
&= \Big(0, \sum_{i=1}^{1} (-1)^{i+1} (a_i+a_{i+1}) , \sum_{i=1}^{2} (-1)^{i+2} (a_i+a_{i+1}), \\
&\hspace{5cm}\ldots , \sum_{i=1}^{n-1} (-1)^{i+(n-1)} (a_i+a_{i+1}) \Big)\\
&=(0, a_2, \ldots ,a_n),
\end{align*}	
and the equality $ T_{\mathcal{R}\to \mathcal{A}} \circ T_{\mathcal{A}\to \mathcal{R}} = {\rm id}$ follows from 
 \begin{align*}
& T_{\mathcal{R}\to \mathcal{A}} \circ T_{\mathcal{A}\to \mathcal{R}}\big( (x_1, x_2, \ldots , x_n)^{\mathcal{A}} \big) \\
&= T_{\mathcal{R}\to \mathcal{A}}  \Big( \Big(0, \sum_{i=1}^{1} (-1)^{i+1} x_i , \sum_{i=1}^{2} (-1)^{i+2} x_i , \ldots , \sum_{i=1}^{n-1} (-1)^{i+(n-1)} x_i \Big) \Big)\\
&=\Big(\sum_{i=1}^{1} (-1)^{i+1} x_i  , \sum_{i=1}^{1} (-1)^{i+1} x_i +  \sum_{i=1}^{2} (-1)^{i+2} x_i , \ldots , \\
&\hspace{2cm} \sum_{i=1}^{n-2} (-1)^{i+(n-2)} x_i +  \sum_{i=1}^{n-1} (-1)^{i+(n-1)} x_i,  \sum_{i=1}^{n-1} (-1)^{i+(n-1)} x_i  \Big)^\A\\
&=  (x_1, x_2, \ldots , x_n)^{\mathcal{A}}.
\end{align*}
 \end{proof}

 \begin{lemma}
 Let $\bm{a}, \bm{b} \in \{0\} \times \mathbb Z_p^{n-1}$ and $ \bm{x}^\mathcal{A}, \bm{y}^\mathcal{A} \in V_{p,n}^{\mathcal{A}}$ such that  $\bm{a} =  T_{\mathcal{A}\to \mathcal{R}} ( \bm{x}^\mathcal{A} )$ and $\bm{b} =  T_{\mathcal{A}\to \mathcal{R}} ( \bm{y}^\mathcal{A} )$. 
Then we have 
 $$
\bm{x}^\mathcal{A} {\sim}^\mathcal{A} \bm{y}^\mathcal{A} \Longleftrightarrow  \bm{a} \sim \bm{b}.
$$
 \end{lemma}
 \begin{proof}
Here we only show that $\bm{x}^\mathcal{A} {\sim}^\mathcal{A} \bm{y}^\mathcal{A} \Longrightarrow  \bm{a} \sim \bm{b}$ and leave the proof of $\bm{x}^\mathcal{A} {\sim}^\mathcal{A} \bm{y}^\mathcal{A} \Longleftarrow  \bm{a} \sim \bm{b}$ which is not used in this paper, to the reader. It suffices to show that when we suppose that $\bm{x}^\A$ and $\bm{y}^\A$ are related by one of the operations (Op1)$^\A$--(Op3)$^\A$ in Definition~\ref{def:Aequiv}, we have $\bm{a} \sim \bm{b}$.  

Suppose that $\bm{y}^\A$ is obtained from $\bm{x}^A$ by (Op1)$^\A$, that is, $\bm{y}^\A = (x_2, \ldots ,x_n, x_1)^\A$. 
We then have 
\begin{align*}
\bm{b} &=&&  T_{\mathcal{A}\to \mathcal{R}} \big( (x_2, \ldots ,x_n, x_1)^\A \big)\\
&=&&\Big(0, \sum_{i=2}^{2} (-1)^{i+2} x_i , \sum_{i=2}^{3} (-1)^{i+3} x_i , \ldots , \sum_{i=2}^{n} (-1)^{i+n} x_i \Big)\\
&\overset{({\rm Op2})}{\longrightarrow} &&
\Big(x_1, \sum_{i=2}^{2} (-1)^{i+2} x_i+ (-1)^2(0-x_1) , \\
&&& \hspace{3cm}\ldots , \sum_{i=2}^{n} (-1)^{i+n} x_i+ (-1)^n(0-x_1) \Big)\\
&=&&\Big(\sum_{i=1}^{1} (-1)^{i+1} x_i, \sum_{i=1}^{2} (-1)^{i+2} x_i, \ldots , \sum_{i=1}^{n} (-1)^{i+n} x_i\Big)\\
&=&&\Big(\sum_{i=1}^{1} (-1)^{i+1} x_i, \sum_{i=1}^{2} (-1)^{i+2} x_i, \ldots , \sum_{i=1}^{n-1} (-1)^{i+(n-1)} x_i, 0 \Big)\\
&\overset{({\rm Op1})^{-1}}{\longrightarrow} &&\Big(0, \sum_{i=1}^{1} (-1)^{i+1} x_i, \sum_{i=1}^{2} (-1)^{i+2} x_i, \ldots , \sum_{i=1}^{n-1} (-1)^{i+(n-1)} x_i\Big) \\
&=&& T_{\mathcal{A}\to \mathcal{R}} \big( (x_1, x_2, \ldots ,x_n)^\A \big)\\
& =&& \bm{a},
\end{align*}
which implies that $\bm{a} \sim \bm{b}$. 

Suppose that $\bm{y}^\A$ is obtained from $\bm{x}^A$ by (Op2)$^\A$, that is, $\bm{y}^\A = (2x-x_1, 2x-x_2, \ldots, 2x-x_n)^\A$ for some $x \in \mathbb Z _p$.
We then have 
\begin{align*}
\bm{b} &=&&  T_{\mathcal{A}\to \mathcal{R}} \big((2x-x_1, 2x-x_2, \ldots, 2x-x_n)^\A\big)\\
&=&&  \Big(0, \sum_{i=1}^{1} (-1)^{i+1} (2x-x_i) ,\ldots , \sum_{i=1}^{n-1} (-1)^{i+(n-1)} (2x-x_i) \Big)\\
&\overset{({\rm Op2})}{\longrightarrow} &&
\Big(x, \sum_{i=1}^{1} (-1)^{i+1} (2x-x_i) +(-1)^2 (0-x) , \\
&&&\hspace{3cm }\ldots , \sum_{i=1}^{n-1} (-1)^{i+(n-1)} (2x-x_i) +(-1)^n (0-x) \Big)\\
&=&&\Big(x,  -\sum_{i=1}^{1} (-1)^{i+1} x_i +x , \ldots , -\sum_{i=1}^{n-1} (-1)^{i+(n-1)} x_i +x \Big)\\
\end{align*}
\begin{align*}
\phantom{b}&\overset{({\rm Op3})}{\longrightarrow}  &&\Big(0, x- \Big(  -\sum_{i=1}^{1} (-1)^{i+1} x_i +x\Big) +0 ,\\
&&& \hspace{3cm } \ldots ,x- \Big(-\sum_{i=1}^{n-1} (-1)^{i+(n-1)} x_i +x \Big) +0 \Big)\\
&=&& \Big(0,\sum_{i=1}^{1} (-1)^{i+1} x_i  , \ldots ,\sum_{i=1}^{n-1} (-1)^{i+(n-1)} x_i \Big)\\
&=&& T_{\mathcal{A}\to \mathcal{R}} \big( (x_1, x_2, \ldots ,x_n)^\A \big)\\
& =&& \bm{a},
\end{align*}
which implies that $\bm{a} \sim \bm{b}$.

Suppose that $\bm{y}^\A$ is obtained from $\bm{x}^A$ by (Op3)$^\A$, that is, $\bm{y}^\A = (x_2, 2x_2-x_1,  x_3, \ldots, x_n)^\A$.
We then have 
\begin{align*}
\bm{b} &=&&  T_{\mathcal{A}\to \mathcal{R}} \big( (x_2, 2x_2-x_1,  x_3, \ldots, x_n)^\A \big)\\
&=&& \Big(0, x_2, -x_2+( 2x_2 -x_1 ) ,  \sum_{i=3}^{3} (-1)^{i+3} x_i + (-1)^{2+3} ( 2x_2 -x_1 ) +  (-1)^{1+3} x_2  ,\\
&&& \ldots ,  \sum_{i=3}^{n-1} (-1)^{i+(n-1)} x_i + (-1)^{2+(n-1)} ( 2x_2 -x_1 ) +  (-1)^{1+(n-1)} x_2  \Big)\\
&=&& \Big(0, x_2, x_2 -x_1 ,  \sum_{i=1}^{3} (-1)^{i+3} x_i ,  \ldots ,  \sum_{i=1}^{n-1} (-1)^{i+(n-1)} x_i \Big)\\
&\overset{({\rm Op4})^{-1}}{\longrightarrow}&& \Big(0, 0+x_2-(x_2 -x_1), x_2 -x_1 ,  \sum_{i=1}^{3} (-1)^{i+3} x_i ,  \ldots ,  \sum_{i=1}^{n-1} (-1)^{i+(n-1)} x_i \Big)\\
&=&& \Big(0, x_1, x_2 -x_1 ,  \sum_{i=1}^{3} (-1)^{i+3} x_i ,  \ldots ,  \sum_{i=1}^{n-1} (-1)^{i+(n-1)} x_i \Big)\\
&=&& T_{\mathcal{A}\to \mathcal{R}} \big( (x_1, x_2, \ldots ,x_n)^\A \big)\\
& =&& \bm{a},
\end{align*}
which implies that $\bm{a} \sim \bm{b}$. 
 \end{proof}
\begin{remark}
The following result can be easily shown from the definition of $T_{\mathcal{R}\to \mathcal{A}}$, Theorem~\ref{main}, and Theorem~\ref{Amain}.

(1) When $p$ is an odd integer, for $\bm{a}=(a_1, a_2) \in U_{p,2}$ such that $a_1=0$, it holds 
\[
T_{\mathcal{R}\to \mathcal{A}}(\bm{a}) \in U_{p,2}^{\A}.
\]

(2) For $\bm{a}=(a_1, a_2) \in \eta_{\varepsilon}$ such that $a_1=0$, it holds 
\[
T_{\mathcal{R}\to \mathcal{A}}(\bm{a}) \in 
\begin{cases}
\eta_{\mu=\mu_p({\bm{a}})}^{\A} &\mbox{ if $\frac{p}{2}$ is odd},\\
\eta_{\tau=p, \varepsilon}^{\A}  &\mbox{ if $\frac{p}{2}$ is even}.
\end{cases}
\]
   
(3) For  $\bm{a}=(a_1, a_2, \ldots, a_n) \in \delta_{\tau}$ such that $a_1=0$, it holds 
\[
T_{\mathcal{R}\to \mathcal{A}}(\bm{a}) \in \delta_{\tau}^{\A}.
\]

(4) For  $\bm{a}=(a_1, a_2, \ldots, a_n) \in \alpha_{\tau,\mu}$ such that $a_1=0$, it holds 
\[
T_{\mathcal{R}\to \mathcal{A}}(\bm{a}) \in \alpha_{\kappa=0,\tau,\mu}^{\A}.
\]

(5) For $\bm{a}=(a_1, a_2, \ldots, a_n) \in \beta_{\tau,\varepsilon}$ such that $a_1=0$, it holds 
\[
T_{\mathcal{R}\to \mathcal{A}}(\bm{a}) \in \beta_{\tau,\varepsilon}^{\A}.
\]

(6) For $\bm{a}=(a_1, a_2, \ldots, a_n) \in \gamma_{\tau,\varepsilon,\mu}$ such that $a_1=0$, it holds 
\[
T_{\mathcal{R}\to \mathcal{A}}(\bm{a}) \in \gamma_{\kappa=0,\tau,\varepsilon,\mu}^{\A}.
\]
\end{remark}

We then have the following lemma for each equivalence class of $U_p$.

\begin{lemma}\label{element}
(1) When $p$ is an odd integer, for $\bm{a} \in U_{p,2}$, it holds 
\[\bm{a}  \sim (0,0).\]

(2) For $\bm{a} \in \eta_{\varepsilon}$, it holds 
\[\bm{a}  \sim \begin{cases}
(0,0) & {\rm if \ } \varepsilon=0,\\
(0,1) & {\rm if \ } \varepsilon=1.\end{cases}
\]
 
(3) For  $\bm{a} \in \delta_{\tau}$, it holds 
\[\bm{a}  \sim (0,0, \tau,-\tau, 2\tau,-2\tau, \ldots,\frac{n-2}{2}\tau, -\frac{n-2}{2}\tau).\]

(4) For  $\bm{a} \in \alpha_{\tau, \mu}$, it holds 
\[\bm{a}  \sim \begin{cases}
(0,0,\tau,-\tau,2\tau,-2\tau,\ldots,\frac{n+\mu-2}{2}\tau,-\frac{n+\mu-2}{2}\tau,\underbrace{0,\tau,\ldots,0,\tau}_{-\mu})  & {\rm if \ } -n<\mu \leq 0,
\\
(0,0,\tau,-\tau,2\tau,-2\tau,\ldots,\frac{n-\mu-2}{2}\tau,-\frac{n-\mu-2}{2}\tau,\underbrace{0,\ldots,0}_{\mu})  & {\rm if \ } 0<\mu < n.\end{cases}\]
 
(5) For  $\bm{a} \in \beta_{\tau, \varepsilon}$, it holds 
\[\bm{a}  \sim \begin{cases}
(0,0, \tau,-\tau, 2\tau,-2\tau, \ldots,\frac{n-2}{2}\tau, -\frac{n-2}{2}\tau) & {\rm if \ } \varepsilon=0,\\
(0,1,\tau,1-\tau,2\tau,1-2\tau,\ldots,\frac{n-2}{2}\tau,1-\frac{n-2}{2}\tau)& {\rm if \ } \varepsilon=1.\end{cases}\]

(6) For  $\bm{a} \in \gamma_{\tau, \varepsilon, \mu}$, it holds 
\[\bm{a}  \sim \begin{cases}
(0,0,\tau,-\tau,2\tau,-2\tau,\ldots,\frac{n-\mu-2}{2}\tau,-\frac{n-\mu-2}{2}\tau,\underbrace{0,\ldots,0}_{\mu})  & {\rm if \ } \varepsilon=0,\\
(0,1,\tau,1-\tau,2\tau,1-2\tau,\ldots,\frac{n-\mu-2}{2}\tau,1-\frac{n-\mu-2}{2}\tau,\underbrace{0,1,\ldots,0,1}_{\mu}) & {\rm if \ } \varepsilon=1.\end{cases}\]
\end{lemma}

\begin{proof}

(1) For $\bm{a} \in U_{p,2}$,
\begin{align*}
&&&\bm{a}=(a_1, a_2)\\
 &\overset{({\rm Op2})}{\longrightarrow}&&(0,a_1+a_2)\\
&\overset{T_{\mathcal{R}\to \mathcal{A}}}{\longrightarrow}&&(a_1+a_2,a_1+a_2)^{\A} \\
&\widesim[4]{Lem.\ref{Alemma}{\A}} & &(0,0)^{\A} \\
&\overset{T_{\mathcal{A}\to \mathcal{R}}}{\longrightarrow}&&(0,0).
\end{align*}

(2) For $\bm{a} \in \eta_{\varepsilon}$,
 \begin{align*}
&&&\bm{a}=(a_1, a_2)\\
 &\overset{({\rm Op2})}{\longrightarrow}&&(0,a_1+a_2)\\
&\overset{T_{\mathcal{R}\to \mathcal{A}}}{\longrightarrow}&&(a_1+a_2,a_1+a_2)^{\A} \in
\begin{cases}
\eta_{\mu=\mu_p(\bm{a})} & \mbox{if $\frac{p}{2}$ is odd}\\
\eta_{\tau=p,\varepsilon} & \mbox{if $\frac{p}{2}$ is even}
\end{cases} \\
&\widesim[4]{Lem.\ref{Alemma}{\A}} & &\begin{cases}
(0,0)^{\A} \mbox{ if } (a_1+a_2,a_1+a_2)^{\A} \in \eta_2 \\
(1,1)^{\A} \mbox{ if } (a_1+a_2,a_1+a_2)^{\A} \in \eta_{-2} \\
(0,0)^{\A} \mbox{ if } (a_1+a_2,a_1+a_2)^{\A} \in \eta_{\tau=p,\varepsilon=0} \\
(1,1)^{\A} \mbox{ if } (a_1+a_2,a_1+a_2)^{\A} \in \eta_{\tau=p,\varepsilon=1} 
\end{cases}\\
&\overset{(\star)}{=}&&\begin{cases}
(0,0)^{\A} & \mbox{ if $\varepsilon=0$}\\
(1,1)^{\A} & \mbox{ if $\varepsilon=1$}
\end{cases} \\
&\overset{T_{\mathcal{A}\to \mathcal{R}}}{\longrightarrow}&&\begin{cases}
(0,0) & \mbox{ if $\varepsilon=0$}\\
(0,1) &\mbox{ if $\varepsilon=1$},
\end{cases} 
\end{align*}
where $\overset{(\star)}{=}$ holds since when $\varepsilon=0$, $(a_1+a_2,a_1+a_2)^{\A} \in \eta_2$ and when $\varepsilon=1$, $(a_1+a_2,a_1+a_2)^{\A} \in \eta_{-2}$.

(3)  For  $\bm{a} \in \delta_{\tau}$,
\begin{align*}
&&&\bm{a}=(a_1, a_2, \ldots, a_n)\\
 &\overset{({\rm Op2})}{\longrightarrow}&&(0,a_2+a_1,a_3-a_1,\ldots,a_{n-1}-a_1,a_n+a_1)\\
&\overset{T_{\mathcal{R}\to \mathcal{A}}}{\longrightarrow}&&(a_1+a_2,a_2+a_3,\ldots,a_{n-1}+a_n,a_n+a_1)^{\A} \\
&\widesim[4]{Lem.\ref{Alemma}{\A}} & &(0, \tau, 0, \tau, \ldots, 0, \frac{2-n}{2}\tau)^{\A} \\
&\overset{T_{\mathcal{A}\to \mathcal{R}}}{\longrightarrow}&&(0,0, \tau,-\tau, 2\tau,-2\tau, \ldots,\frac{n-2}{2}\tau, -\frac{n-2}{2}\tau).
\end{align*}

(4) For  $\bm{a} \in \alpha_{\tau, \mu}$,
\begin{align*}
&&&\bm{a}=(a_1, a_2, \ldots, a_n)\\
 &\overset{({\rm Op2})}{\longrightarrow}&&(0,a_2+a_1,a_3-a_1,\ldots,a_{n-1}-a_1,a_n+a_1)\\
&\overset{T_{\mathcal{R}\to \mathcal{A}}}{\longrightarrow}&&(a_1+a_2,a_2+a_3,\ldots,a_{n-1}+a_n,a_n+a_1)^{\A} \\
&\widesim[4]{Lem.\ref{Alemma}{\A}} & &\begin{cases}
(0, \tau, 0, \tau, \ldots, 0, \frac{2-n-\mu}{2}\tau,\underbrace{\tau,\ldots,\tau}_{-\mu})^{\A}  & {\rm if \ } -n<\mu \leq 0,\\
(0, \tau, 0, \tau, \ldots, 0, \frac{2-n+\mu}{2}\tau,\underbrace{0,\ldots,0}_{\mu}
)^{\A}  & {\rm if \ } 0<\mu < n
\end{cases} \\
&\overset{T_{\mathcal{A}\to \mathcal{R}}}{\longrightarrow}&&\begin{cases}
(0,0,\tau,-\tau,\ldots,\frac{n+\mu-2}{2}\tau,-\frac{n+\mu-2}{2}\tau,\underbrace{0,\tau,\ldots,0,\tau}_{-\mu})  & {\rm if \ } -n<\mu \leq 0,
\\
(0,0,\tau,-\tau,\ldots,\frac{n-\mu-2}{2}\tau,-\frac{n-\mu-2}{2}\tau,\underbrace{0,\ldots,0}_{\mu})  & {\rm if \ } 0<\mu < n.
\end{cases}
\end{align*}

(5) For  $\bm{a} \in \beta_{\tau, \varepsilon}$,
\begin{align*}
&&&\bm{a}=(a_1, a_2, \ldots, a_n)\\
 &\overset{({\rm Op2})}{\longrightarrow}&&(0,a_2+a_1,a_3-a_1,\ldots,a_{n-1}-a_1,a_n+a_1)\\
&\overset{T_{\mathcal{R}\to \mathcal{A}}}{\longrightarrow}&&(a_1+a_2,a_2+a_3,\ldots,a_{n-1}+a_n,a_n+a_1)^{\A} \\
&\widesim[4]{Lem.\ref{Alemma}{\A}} & &\begin{cases}
(0,\tau,0,\tau,\ldots,0,\frac{2-n}{2}\tau)^{\A} & {\rm if \ } \varepsilon=0,\\
(1,\tau+1,1,\tau+1,\ldots,1,\frac{2-n}{2}\tau+1)^{\A}& {\rm if \ } \varepsilon=1\end{cases} \\
&\overset{T_{\mathcal{A}\to \mathcal{R}}}{\longrightarrow}&&\begin{cases}
(0,0, \tau,-\tau, 2\tau,-2\tau, \ldots,\frac{n-2}{2}\tau, -\frac{n-2}{2}\tau) & {\rm if \ } \varepsilon=0,\\
(0,1,\tau,1-\tau,2\tau,1-2\tau,\ldots,\frac{n-2}{2}\tau,1-\frac{n-2}{2}\tau)& {\rm if \ } \varepsilon=1.
\end{cases}
\end{align*}

(6) For  $\bm{a} \in \gamma_{\tau, \varepsilon, \mu}$,
\begin{align*}
&&&\bm{a}=(a_1, a_2, \ldots, a_n)\\
 &\overset{({\rm Op2})}{\longrightarrow}&&(0,a_2+a_1,a_3-a_1,\ldots,a_{n-1}-a_1,a_n+a_1)\\
&\overset{T_{\mathcal{R}\to \mathcal{A}}}{\longrightarrow}&&(a_1+a_2,a_2+a_3,\ldots,a_{n-1}+a_n,a_n+a_1)^{\A} \\
&\widesim[4]{Lem.\ref{Alemma}{\A}} & &\begin{cases}
(0,\tau,0,\tau,\ldots,0,\frac{2-n+\mu}{2}\tau,\underbrace{0,\ldots,0}_{\mu})^{\A} & {\rm if \ } \varepsilon=0,\\
(1,\tau+1,1,\tau+1,\ldots,1,\frac{2-n+\mu}{2}\tau+1,\underbrace{1,\dots,1}_{\mu})^{\A}& {\rm if \ } \varepsilon=1\end{cases} \\
&\overset{T_{\mathcal{A}\to \mathcal{R}}}{\longrightarrow}&&\begin{cases}
(0,0,\tau,-\tau,2\tau,-2\tau,\ldots,\frac{n-\mu-2}{2}\tau,-\frac{n-\mu-2}{2}\tau,\underbrace{0,\ldots,0}_{\mu})  & {\rm if \ } \varepsilon=0,\\
(0,1,\tau,1-\tau,2\tau,1-2\tau,\ldots,\frac{n-\mu-2}{2}\tau,1-\frac{n-\mu-2}{2}\tau,\underbrace{0,1,\ldots,0,1}_{\mu}) & {\rm if \ } \varepsilon=1.\end{cases}
\end{align*}

\end{proof}

\begin{lemma}\label{keylemma1}
(1) Let $p$ be an even integer. We have
\[
U_{p,2}=  \eta_0 \sqcup \eta_1,
\]
that is, $U_{p,2}$ can be split into the disjoint sets $\eta_0$ and $\eta_1$ defined in Theorem~\ref{main} by $\varepsilon_p$.

(2) Let $n$ be an even integer greater than 2 and $p$ an even integer. Then, $U_{p,n}$ is equal to 
$$\Bigg( \bigsqcup_{\tiny
\begin{minipage}{3cm}
$\tau \in \{1,\ldots,p\} 
\mbox{ s.t. } (\tau|p) \\ \wedge \big(\tau \equiv 1 {\rm \  (mod \ 2)} \big),\\
\mu \in \mathbb{Z} \mbox{ s.t. } (-n<\mu<n) \wedge\\ \big(\mu \equiv 0 {\rm \  (mod \ 2)}\big) \wedge \\ \big(\frac{n-|\mu|}{2} \equiv 0 {\rm \  (mod \ 2)}\big)$
\end{minipage}
} \hspace{-1em}\alpha_{\tau, \mu} \Bigg)
\sqcup
\Bigg( \bigsqcup_{\tiny
\begin{minipage}{3cm}
$\tau \in \{1,\ldots,p\} \mbox{ s.t. } (\tau|p) \\
\wedge \big(\tau \equiv 0 {\rm \  (mod \ 2)}\big) \wedge \\\big( \frac{p}{\tau} \equiv 1 {\rm \  (mod \ 2)}\big),\\
\varepsilon \in \{0,1\}$
\end{minipage}
}  \hspace{-1em}\beta_{\tau, \varepsilon} \Bigg)
\sqcup
\Bigg( \bigsqcup_{\tiny
\begin{minipage}{3cm}
$\tau \in \{1,\ldots,p\} \mbox{ s.t. } (\tau|p) \\ \wedge \big(\tau \equiv 0 {\rm \  (mod \ 2)}\big)\\  \wedge \big( \frac{p}{\tau} \equiv 0 {\rm \  (mod \ 2)}\big),\\
\varepsilon \in \{0,1\},\\
\mu \in \mathbb{Z} \mbox{ s.t. } (0 \leq \mu<n) \\ \wedge \big(\mu \equiv 0 {\rm \  (mod \ 2)}\big) \\ \wedge \big(\frac{n-|\mu|}{2} \equiv 0 {\rm \  (mod \ 2)}\big)$
\end{minipage}
}  \hspace{-1em}\gamma_{\tau, \varepsilon, \mu} \Bigg)
,$$
 that is, $U_{p,n}$ can be split into the disjoint sets $\alpha_{\tau, \mu}$s, $\beta_{\tau, \varepsilon}$s, and $\gamma_{\tau, \varepsilon, \mu}$s defined in Theorem~\ref{main} by $\tau_p$, $\varepsilon_p$, $\mu_p$, and $\mu_{p, \tau}$.

\end{lemma}

\begin{proof}
(1) By the definition of $\varepsilon_p$, it is clear that $U_{p,2}=  \eta_0 \sqcup \eta_1  \sqcup \eta_\infty$, and hence, it suffices to show that $\eta_\infty =\emptyset$, where $\eta_\infty =\{\bm{a} \in U_{p,2} \mid \varepsilon_p(\bm{a})=\infty \}$. 
For $\bm{a} = (a_1, a_2)\in U_{p,2}$, since $a_1+ a_2 \equiv a_2 + a_1 \equiv 0 \pmod{2}$ or $a_1+ a_2 \equiv a_2 + a_1 \equiv 1 \pmod{2}$ holds, $\varepsilon_p(\bm{a})=0$ or $1$, which implies $\eta_\infty =\emptyset$. 

(2) 
By the definition of $\tau_p$, $U_{p,n}$ can be split as 
$$\Bigg( \bigsqcup_{\tiny
\begin{minipage}{3cm}
$\tau \in \{1,\ldots,p\} 
\mbox{ s.t. } (\tau|p) \\ \wedge \big(\tau \equiv 1 {\rm \  (mod \ 2)} \big),\\
$
\end{minipage}
} \hspace{-1.5em}\alpha_{\tau} \Bigg)
\sqcup
\Bigg( \bigsqcup_{\tiny
\begin{minipage}{3cm}
$\tau \in \{1,\ldots,p\} \mbox{ s.t. } (\tau|p) \\
\wedge \big(\tau \equiv 0 {\rm \  (mod \ 2)}\big) \wedge \\\big( \frac{p}{\tau} \equiv 1 {\rm \  (mod \ 2)}\big),\\
$
\end{minipage}
}  \hspace{-1.5em}\beta_{\tau} \Bigg)
\sqcup
\Bigg( \bigsqcup_{\tiny
\begin{minipage}{3cm}
$\tau \in \{1,\ldots,p\} \mbox{ s.t. } (\tau|p) \\ \wedge \big(\tau \equiv 0 {\rm \  (mod \ 2)}\big)\\  \wedge \big( \frac{p}{\tau} \equiv 0 {\rm \  (mod \ 2)}\big),\\
$
\end{minipage}
}  \hspace{-1.5em}\gamma_{\tau} \Bigg)
,$$
where 
\begin{align*}
&\alpha_{\tau}=\{\bm{a} \in U_{p,n} \mid \tau_p(\bm{a})=\tau\},\\
&\beta_{\tau}=\{\bm{a} \in U_{p,n} \mid \tau_p(\bm{a})=\tau\}, \mbox{and}\\
&\gamma_{\tau}=\{\bm{a} \in U_{p,n} \mid \tau_p(\bm{a})=\tau \}.
\end{align*}

We now focus on $\alpha_\tau$ for some $\tau$. 
For $\bm{a}=(a_1 ,\ldots , a_n) \in \alpha_\tau$, by the definition of $\mu_p$, it is clear that $\mu_p (\bm{a}) \equiv 0 \pmod{2}$ and $-n\leq \mu_p (\bm{a})\leq n$. 
When we assume that $\mu_p (\bm{a})= -n$  or $n$, it holds that 
$$
a_1+a_2 \equiv \cdots = a_n + a_1 \pmod{2}.
$$
This implies $\tau_p (\bm{a})$ is an even number, which contradicts the condition of $\tau$ for $\alpha_\tau$. Thus we have $\mu_p (\bm{a})\not= \pm n$. 
Besides, since  $O\big((a_1+ a_2,  \ldots , a_n+a_1)\big)$, that is the number of odd entries in $(a_1+ a_2, \ldots , a_n+a_1)$, is even, $O\big((a_1+ a_2,  \ldots , a_n+a_1)\big)=2m$ for some integer $m>0$.
We then have $\mu_p(\bm{a})= (n-2m) -2m=n-4m$ and 
\begin{align*}
\frac{n-|\mu_p(\bm{a})|}{2} = \frac{n-|n-4m|}{2} = 2m \mbox{ or } n-2m \equiv 0 \pmod{2},
\end{align*}
where we note that $n$ is even.
Therefore, we have  
\[
\alpha_\tau = \bigsqcup_{\tiny
\begin{minipage}{2.9cm}
$\mu \in \mathbb{Z} \mbox{ s.t. } (-n<\mu<n) \wedge \big(\mu \equiv 0 {\rm \  (mod \ 2)}\big) \wedge \big(\frac{n-|\mu|}{2} \equiv 0 {\rm \  (mod \ 2)}\big)$
\end{minipage}
} \hspace{-1em}\alpha_{\tau, \mu}. 
\]

Next we focus on $\beta_\tau$ for some $\tau$. 
For $\bm{a}=(a_1 ,\ldots , a_n) \in \beta_\tau$, since $\tau \equiv 0 \pmod{2}$, we have 
$a_1+a_2 \equiv \cdots \equiv a_n+a_1 \pmod{2}$.
This implies that $\varepsilon_p(\bm{a}) =0$ or $1$, and hence, $\varepsilon_p(\bm{a})\not= \infty$.
Therefore, we have  
\[
\beta_\tau =\bigsqcup_{\tiny
\begin{minipage}{2cm}
\begin{center}
$\varepsilon \in \{0,1\}$
\end{center}
\end{minipage}
}  \hspace{-1em}\beta_{\tau, \varepsilon}.
\]

Finally, let us consider $\gamma_\tau$ for some $\tau$.
For ${\bm{a}} \in \gamma_\tau$, by the same reason as in the case of $\beta_\tau$, we have $\varepsilon_p(\bm{a})\not = \infty$.
Besides, by the definition of $\mu_{p,\tau}$ and the property of $\mu_{\frac{p}{\tau}}$ mentioned in the case of $\alpha_\tau$, we have $0\leq |\mu_{p,\tau} (\bm{a})| <n$, $ \mu_{p,\tau} (\bm{a}) \equiv 0\pmod{2}$, and $\frac{n-|\mu_{p,\tau}(\bm{a})|}{2}  \equiv 0 \pmod{2}$.
Therefore, we have  
\[
\gamma_\tau =\bigsqcup_{\tiny
\begin{minipage}{3cm}
$
\varepsilon \in \{0,1\},\\
\mu \in \mathbb{Z} \mbox{ s.t. } (0 \leq \mu<n) \\ \wedge \big(\mu \equiv 0 {\rm \  (mod \ 2)}\big) \\ \wedge \big(\frac{n-|\mu|}{2} \equiv 0 {\rm \  (mod \ 2)}\big)
$
\end{minipage}
}  \hspace{-1.5em}\gamma_{\tau,\varepsilon,\mu}
\]

As a consequence, we have the equality of (2) of this lemma.
\end{proof}

\noindent{\bf Proof of Theorem~\ref{main}.} 
First, we remark that each equivalence class is distinguished from the others by using the maps $\varepsilon_p$, $\tau_p$, $\mu_p$, and $\mu_{p,\tau}$. 

In the case that $n=2$ and $p$ is an odd integer, by (1) of  Lemma~\ref{element}, there exists a unique equivalence class, which is $U_{p,2}$. Thus we have the equality of (i) of (1).

In the case that  $n=2$ and $p$ is an even integer, by Lemma~\ref{keylemma1}, we have $U_{p,2} = \eta_0 \sqcup \eta_1$, that is, $U_{p,2}$ can be split into the disjoint sets $\eta_0$ and $\eta_1$ by $\varepsilon_p$. Besides, (2) of Lemma~\ref{element} implies that each of $\eta_0$ and $\eta_1$ is an equivalence class. Thus we have two equivalence classes $\eta_0$ and $\eta_1$, and we have the equality of (ii) of (1).

In the case that  $n$ is an even integer greater than $2$ and $p$ is an odd integer, by the definition of $\tau_p$, we have $U_{p,n} = \bigsqcup_{\tau \in \{1, \ldots , p\} \mbox{ s.t. } \tau|p } \delta_\tau$, that is, $U_{p,n}$ can be split into the disjoint sets $\delta_\tau$s by $\tau_p$. Besides, (3) of Lemma~\ref{element} implies that each of $\delta_\tau$s is an equivalence class. Thus we have the equality of (i) of (2).

In the case that  $n$ is an even integer greater than $2$ and $p$ is an even integer, by Lemma~\ref{keylemma1}, $U_{p,n} $ is equal to  
$$\Bigg( \bigsqcup_{\tiny
\begin{minipage}{3cm}
$\tau \in \{1,\ldots,p\} 
\mbox{ s.t. } (\tau|p) \\ \wedge \big(\tau \equiv 1 {\rm \  (mod \ 2)} \big),\\
\mu \in \mathbb{Z} \mbox{ s.t. } (-n<\mu<n) \wedge\\ \big(\mu \equiv 0 {\rm \  (mod \ 2)}\big) \wedge \\ \big(\frac{n-|\mu|}{2} \equiv 0 {\rm \  (mod \ 2)}\big)$
\end{minipage}
} \hspace{-1em}\alpha_{\tau, \mu} \Bigg)
\sqcup
\Bigg( \bigsqcup_{\tiny
\begin{minipage}{3cm}
$\tau \in \{1,\ldots,p\} \mbox{ s.t. } (\tau|p) \\
\wedge \big(\tau \equiv 0 {\rm \  (mod \ 2)}\big) \wedge \\\big( \frac{p}{\tau} \equiv 1 {\rm \  (mod \ 2)}\big),\\
\varepsilon \in \{0,1\}$
\end{minipage}
}  \hspace{-1em}\beta_{\tau, \varepsilon} \Bigg)
\sqcup
\Bigg( \bigsqcup_{\tiny
\begin{minipage}{3cm}
$\tau \in \{1,\ldots,p\} \mbox{ s.t. } (\tau|p) \\ \wedge \big(\tau \equiv 0 {\rm \  (mod \ 2)}\big)\\  \wedge \big( \frac{p}{\tau} \equiv 0 {\rm \  (mod \ 2)}\big),\\
\varepsilon \in \{0,1\},\\
\mu \in \mathbb{Z} \mbox{ s.t. } (0 \leq \mu<n) \\ \wedge \big(\mu \equiv 0 {\rm \  (mod \ 2)}\big) \\ \wedge \big(\frac{n-|\mu|}{2} \equiv 0 {\rm \  (mod \ 2)}\big)$
\end{minipage}
}  \hspace{-1em}\gamma_{\tau, \varepsilon, \mu} \Bigg)
,$$
 that is, $U_{p,n}$ can be split into the disjoint sets $\alpha_{\tau, \mu}$s, $\beta_{\tau, \varepsilon}$s, and $\gamma_{\tau, \varepsilon, \mu}$s by $\tau_p$, $\varepsilon_p$, $\mu_p$, and $\mu_{p, \tau}$. Besides, (4)-(6) of Lemma~\ref{element} imply that each of $\alpha_{\tau, \mu}$s, $\beta_{\tau, \varepsilon}$s, and $\gamma_{\tau, \varepsilon, \mu}$s is an equivalence class. Thus we have the equality of (ii) of (2).
\qed

\section{Invariants of unoriented spatial graphs using $\mathcal{R}$-palettes }
A \textit{spatial graph} is a graph embedded in the 3-dimensional Euclidean space $\mathbb{R}^3$. Two spatial graphs are \textit{equivalent} if one can be deformed by an ambient isotopy into the other. A \textit{spatial graph diagram} $D$ of a spatial graph $G$ is an image of $G$ by a regular projection onto the 2-dimensional Euclidean space $\mathbb{R}^2$ with a crossing information at each double point.
It is well-known that two spatial graph diagrams represent an equivalent spatial graph if and only if they are related by a finite sequence of the generalized Reidemeister moves as in Fig.~\ref{rmoves}. 

\begin{figure}[h]
  \begin{center}
    \includegraphics[clip,width=10.5cm]{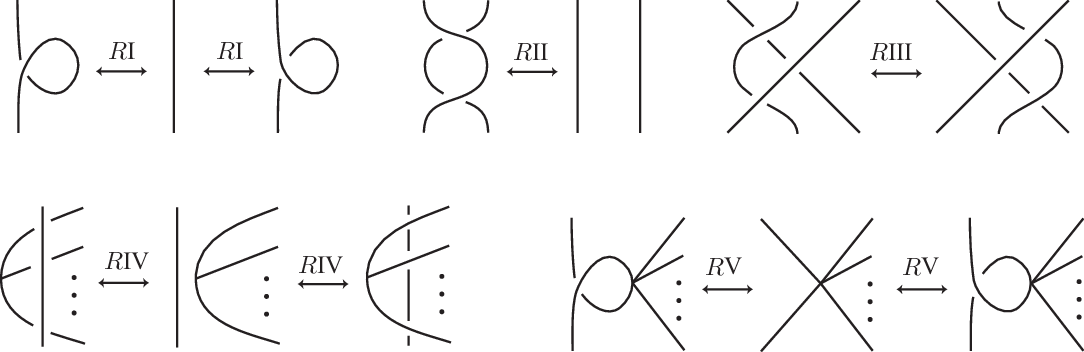}
    \caption{The generalized Reidemeister moves}
    \label{rmoves}
  \end{center}
\end{figure}

Each edge of a spatial graph diagram is separated into some pieces. These pieces are said to be \textit{arcs}. The 2-dimensional Euclidean space $\mathbb{R}^2$ is separated into some pieces by edges of a spatial graph diagram. These pieces are said to be \textit{regions}. 
In this section, a spatial graph means an unoriented spatial graph.
\begin{definition}\label{unorienteddef}
Let $P$ be an $\mathcal{R}$-palette for Dehn $p$-colorings.
Let $D$ be a diagram of an unoriented spatial graph and $\mathcal{R}(D)$ the set of regions of $D$.
A {\it Dehn $(p, P)$-coloring} of $D$ is a map $C: \mathcal{R}(D) \to \mathbb{Z}_p$ satisfying the following conditions: 
\begin{itemize}
\item For a crossing $c$ with regions $r_1, r_2, r_3$, and $r_4$ such that $r_2$ is adjacent to an arbitrary chosen $r_1$ by an under-arc and $r_3$ is adjacent to $r_1$ by the over-arc as depicted in Fig.~\ref{coloring7},
\[
C(r_1)- C(r_2)+ C(r_3) -C(r_4) = 0
\]
holds, which we call the {\it crossing condition}.
\item For a vertex $v$ with regions $r_1, \ldots , r_n$ that appear clockwisely as depicted in Fig.~\ref{coloring7}, 
\[
\Big(C(r_1), C(r_2), \ldots , C(r_n)\Big) \in P
\]
holds, which we call the {\it vertex condition}.
\end{itemize}
\begin{figure}[h]
  \begin{center}
    \includegraphics[clip,width=9cm]{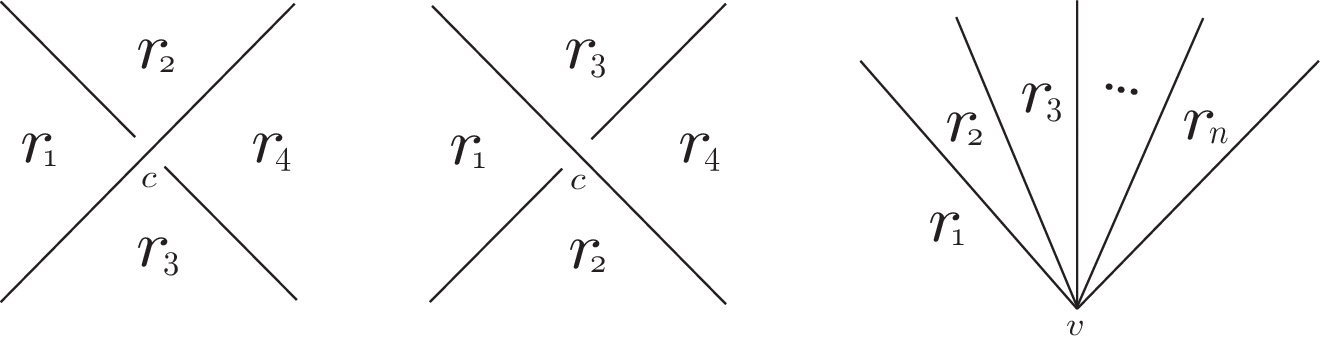}
    \caption{The coloring conditions}
    \label{coloring7}
  \end{center}
\end{figure}
We call $C(r)$ the {\it color} of a region $r$. We denote by ${\rm Col}_{(p ,P)}(D)$ the set of Dehn $(p ,P)$-colorings of $D$.
We denote by $(D,C)$ a diagram $D$ equipped with a Dehn $(p ,P)$-coloring $C$.

\end{definition}
\begin{proposition}\label{prop:coloring}
Let $D$ and $D'$ be diagrams of spatial graphs. If $D$ and $D'$ represent the same spatial graph, then there exists a bijection between ${\rm Col}_{(p ,P)}(D)$ and ${\rm Col}_{(p ,P)}(D')$.  
\end{proposition}

\begin{proof}
Let $D$ and $D'$ be diagrams such that $D'$ is obtained from $D$ by a single generalized Reidemeister move. Let $E$ be a $2$-disk in which the move is applied. Let $C$ be a Dehn
$(p, P)$-coloring of $D$. We define a Dehn $(p, P)$-coloring $C'$ of $D'$, corresponding to
$C$, by $C'(r) = C(r)$ for each region $r$ appearing in the outside of $E$. Then the colors
of the regions appearing in $E$, by $C'$, are uniquely determined, see Fig.~\ref{coloring8} and \ref{coloring9} for generalized Reidemeister moves of type IV and V, respectively.

\begin{figure}[h]
  \begin{center}
    \includegraphics[clip,width=8cm]{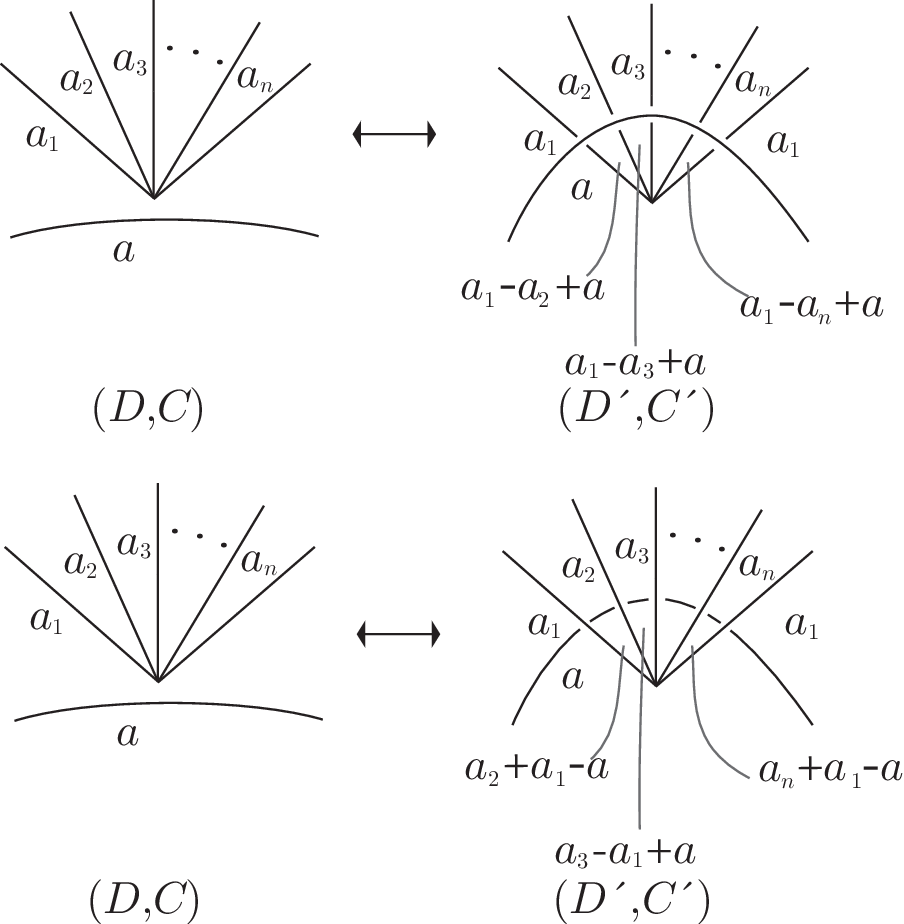}
    \caption{The correspondence of Dehn $(p, P)$-colorings under $R{\rm IV}$}
    \label{coloring8}
  \end{center}
\end{figure}
\begin{figure}[h]
  \begin{center}
    \includegraphics[clip,width=8cm]{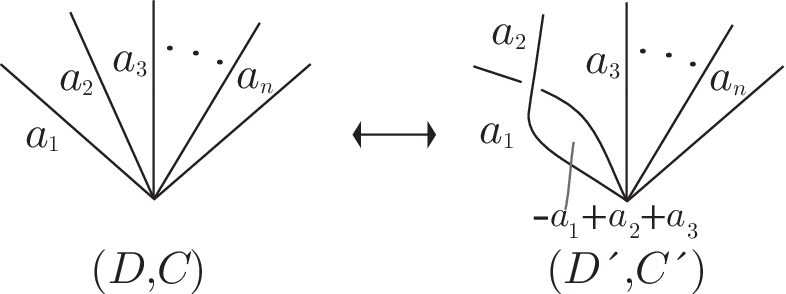}
    \caption{The correspondence of Dehn $(p, P)$-colorings under $R{\rm V}$}
    \label{coloring9}
  \end{center}
\end{figure}

More precisely, for the upper move in Fig.~\ref{coloring8}, we have $(a_1, \ldots , a_n) \in P$ by the vertex condition of the Dehn $(p, P)$-coloring $C$. We then have $(a, a_1-a_2+a, \ldots, a_1 -a_n +a  ) \in P$ by Definition~\ref{def:DihedralRpalettes}, which implies that the vertex condition of the Dehn $(p, P)$-coloring $C'$ for the corresponding vertex is also satisfied. Similar arguments apply to the other cases.

\end{proof}
\noindent Proposition~\ref{prop:coloring} shows that the number of Dehn $(p ,P)$-colorings, i.e. $\# {\rm Col}_{(p ,P)}(D)$, is an invariant of unoriented spatial graphs.

The next example implies that we might be able to distinguish spatial graphs with selecting an appropriate palette.

\begin{example}
We cannot distinguish the spatial graphs $G$ and $H$ with the number of Dehn $(3 ,U_3)$-colorings because it holds that $\# {\rm Col}_{(3 ,U_3)}(G)=\# {\rm Col}_{(3 ,U_3)}(H)=3^5=243$ as depicted in Fig.~\ref{GH}, where $a, b, c, d, e \in \mathbb{Z}_3$. See Example \ref{universal} for $U_3$.

\begin{figure}[h]
  \begin{center}
    \includegraphics[clip,width=10cm]{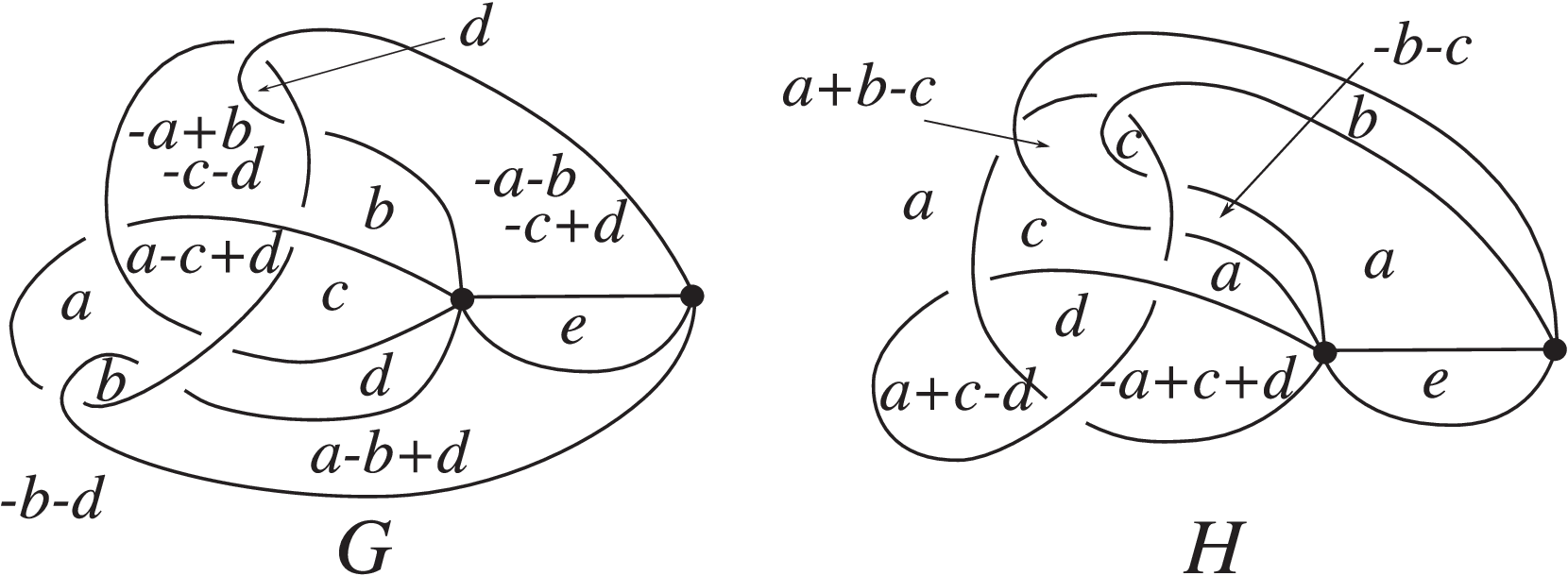}
    \caption{The Dehn $(3, U_3)$-colorings for $G$ and $H$}
    \label{GH}
  \end{center}
\end{figure}

To distinguish these spatial graphs, let us replace $U_3$ with $P=(U_{3,4} \setminus A_{3,4}) \cup A_{3,6}$. See Examples \ref{universal} and \ref{alternating} for $U_{p,n}$ and $A_{p,n}$, respectively. We first compute the number of the Dehn $(3, P)$-colorings of $G$; see Fig.~\ref{G}. Let $C(r_1)=C(r_3)=C(r_5)=a$ and $C(r_2)=C(r_4)=C(r_6)=b$ because the 6-tuples in $P$ are alternating such as $(a, b, a, b, a, b)$ for $a, b \in \mathbb{Z}_3$. When we put $C(r_7)=c$ for $c \in \mathbb{Z}_3$, $C(r_8)=a-b+c$ is given from the crossing condition at $c_1$. Similarly, we have the following colors from each crossing condition:
\[
\begin{cases}
 C(r_9)=2b-c & {\rm from \ the \ crossing \ condition \ at \ }  c_2, \\
 C(r_{10})=a+b-c & {\rm from \ the \ crossing \ condition \ at \ }  c_3,\\ 
 C(r_{11})=2b-c & {\rm from \ the \ crossing \ condition \ at \ }   c_4, \\
 C(r_{12})=c & {\rm from \ the \ crossing \ condition \ at \ }   c_5.
   \end{cases}
\]
Then the crossing conditions at $c_6$ and $c_7$ give $b=c$. This means that the 4-tuples of colors around the 4-valent vertex are alternating as $(C(r_1), C(r_6), C(r_5), C(r_{12}))=(a, b, a, b)$, which does not satisfy the vertex condition. Therefore, we have \[\# {\rm Col}_{(3, P)}(G)=0.\]

\begin{figure}[h]
   \begin{center}
   \includegraphics[width=6cm]{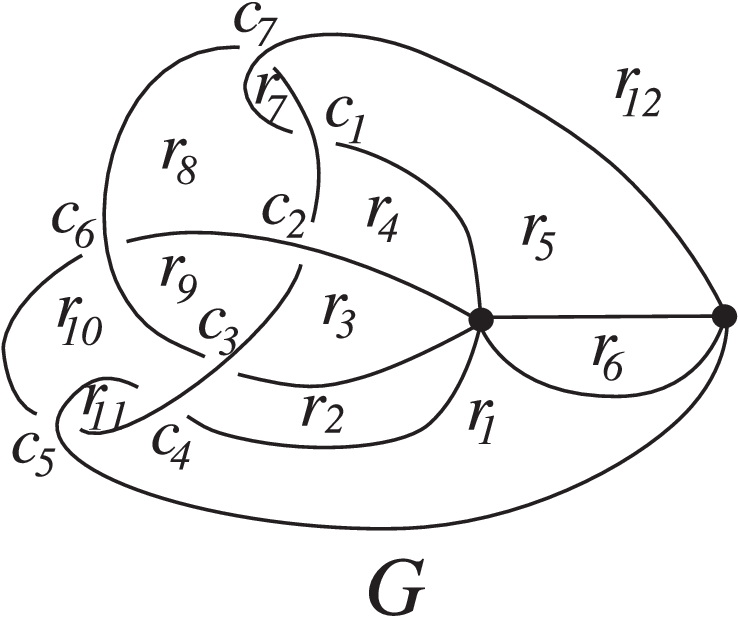}
  \end{center}
  \caption{The spatial graph $G$}
  \label{G} 
  \end{figure}

Next, we compute the number of  Dehn $(3, P)$-colorings of $H$; see Fig.~\ref{H}. As in the case of $G$, let $C(r_1)=C(r_3)=C(r_5)=a$, $C(r_2)=C(r_4)=C(r_6)=b$, and $C(r_7)=c$ for $a, b, c \in \mathbb{Z}_3$. We have the following colors from each crossing condition:
\[
\begin{cases}
 C(r_8)=a-b+c &{\rm from \ the \ crossing \ condition \ at \ }  c_1, \\
 C(r_9)=c & {\rm from \ the \ crossing \ condition \ at \ }  c_2, \\
 C(r_{10})=a+b-c & {\rm from \ the \ crossing \ condition \ at \ }  c_3,\\ 
 C(r_{11})=2b-c & {\rm from \ the \ crossing \ condition \ at \ }  c_4, \\
 C(r_{12})=-b+2c & {\rm from \ the \ crossing \ condition \ at \ }  c_5.
   \end{cases}
\]

Here we observe the 4-valent vertex. The 4-tuples of colors around the 4-valent vertex $(C(r_1), C(r_6), C(r_5), C(r_{12}))=(a, b, a, 2c-b)$ can not be alternating to satisfy the vertex condition, and thus, $b \neq c$ is required. This follows
\[\# {\rm Col}_{(3, P)}(H)=3 \times 3 \times 2=18.\] 

  \begin{figure}[h]
  \begin{center}
   \includegraphics[width=6cm]{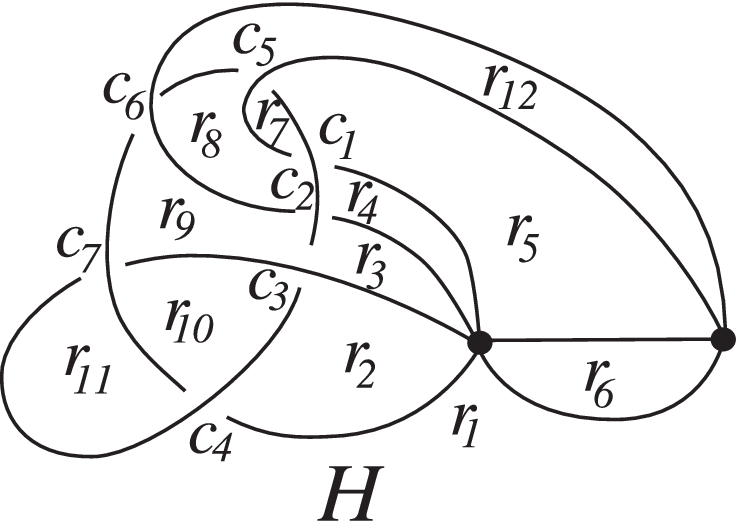}
  \end{center}
  \caption{The spatial graph $H$}
  \label{H}
\end{figure}

Therefore, we see that the spatial graphs $G$ and $H$ are not equivalent.

\end{example}

\section{Generalized $\mathcal{R}$-palettes for region colorings of oriented spatial graph diagrams}\label{generaloriented}

The notion of $\mathcal{R}$-palettes for Dehn $p$-colorings can be extended for knot-theoretic ternary-quasigroups and  region colorings of ``oriented" spatial graph diagrams in general. 
In this section, we will show a generalization of the notion.

\begin{definition}\label{def:horizontal}
 A {\it knot-theoretic ternary-quasigroup} \cite{Niebrzydowski0, Niebrzydowski1, Niebrzydowski2} is a pair of  a   set $X$  and a ternary operation $[\, ]: X^3 \to X; (a,b,c) \mapsto [a,b,c]$ satisfying the following property:
\begin{enumerate}
\item[] \hspace{-0.5cm}($\mathcal{KTQ}$1) For any $a,b,c\in X$, 
\begin{itemize}
\item[(i)] there exists a unique $d_1\in X$ such that $[a,b,d_1]=c$, 
\item[(ii)] there exists a unique $d_2\in X$ such that $[a,d_2,b]=c$, 
\item[(iii)] there exists a unique $d_3 \in X$ such that $[d_3,a,b]=c$.
\end{itemize}
\item[] \hspace{-0.5cm}($\mathcal{KTQ}$2) For any $a,b,c,d \in X$, it holds that 
\[
\begin{array}{l}
[b,[a,b,c],[a,b,d]] = [c,[a,b,c],[a,c,d]]=[d,[a,b,d],[a,c,d]].
\end{array} 
\]
\end{enumerate}

\end{definition}
\noindent The axioms of a knot-theoretic ternary-quasigroup $(X, [\,])$ are obtained from the oriented Reidemeister moves of link diagrams, which is observed when we consider a region coloring of an oriented link diagram by $(X, [\,])$, that is, an assignment of an element of $X$ to each region satisfying the crossing condition depicted in Fig.~\ref{coloring2}.
See Fig.~\ref{RmoveIII} for the correspondence between the Reidemeister move of type III and the axiom ($\mathcal{KTQ}$2) of a knot-theoretic ternary-quasigroup $(X, [\,])$.
\begin{figure}[h]
  \begin{center}
    \includegraphics[clip,width=6.0cm]{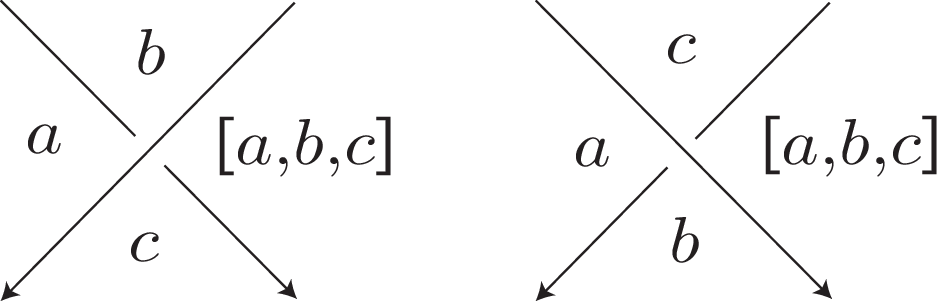}
    \caption{The crossing condition for a region coloring by $(X,[\,])$}
    \label{coloring2}
  \end{center}
\end{figure}

\begin{figure}[h]
  \begin{center}
      \includegraphics[clip,width=5.6cm]{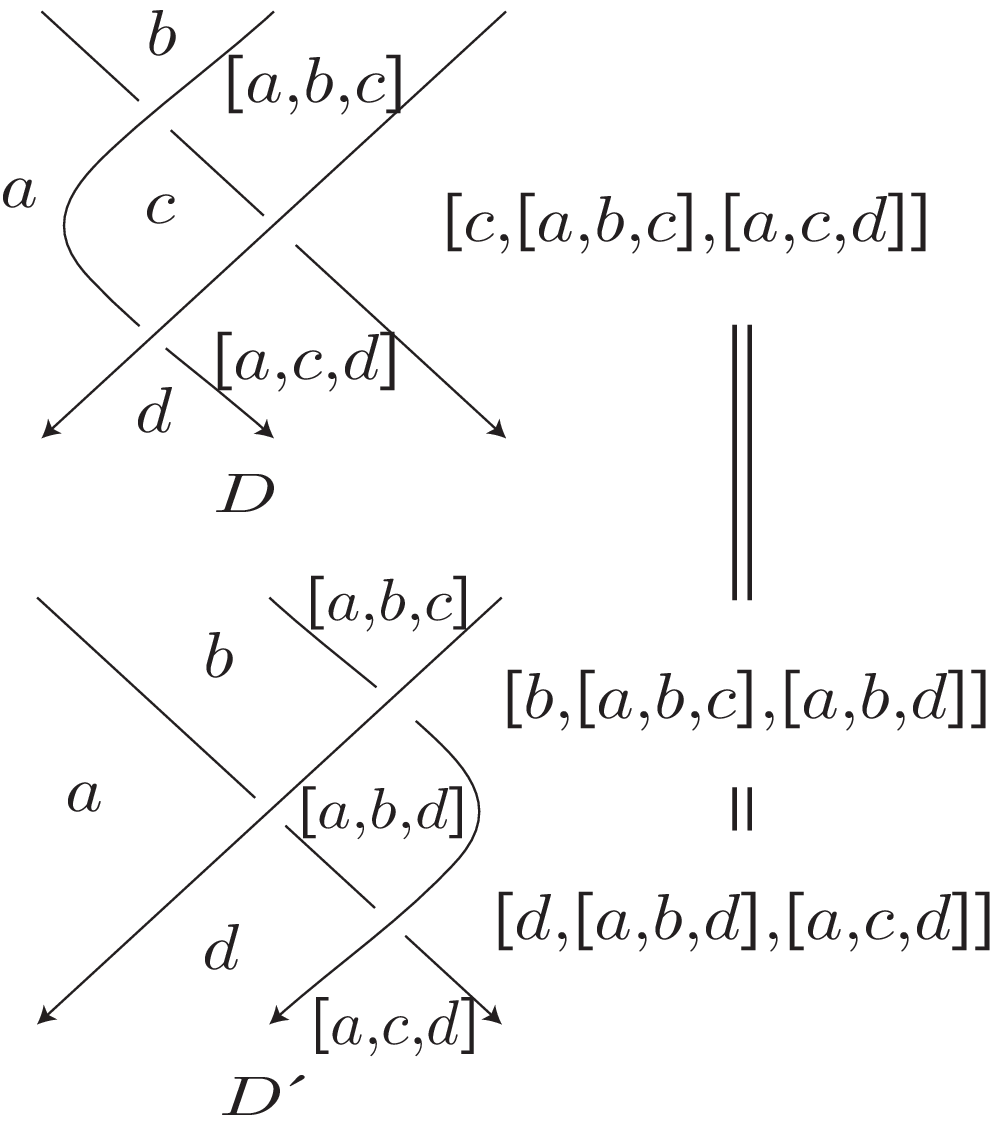}
    \caption{The correspondence between $R$III and ($\mathcal{KTQ}$2)}
    \label{RmoveIII}
  \end{center}
\end{figure}

Let $(X, [\,])$ be a knot-theoretic ternary-quasigroup. 
Define $H_{(a,b;i)}:X \to X \ (i \in \{1,2,3\})$ by $$H_{(a, b;1)}(c)=[c,a,b], \  H_{(a, b;2)}(c)=[a,c,b], \mbox{ and } \ H_{(a, b;3)}(c)=[a,b,c]$$ for $a,b,c\in X$.

\begin{definition}\label{def:LR}
For elements $a, b, c \in X$,  we define maps $\underline L_{(a,b)}$, $\overline L_{(a,b)}$, $\underline R_{(a,b)}$, and $\overline R_{(a,b)}: X \to X$ by
$$\underline L_{(a,b)}(c)=H_{(a, b;2)}(c), \ \overline L_{(a,b)}(c)=H_{(a, b;3)}(c),$$
$$\underline R_{(a,b)}(c)=H^{-1}_{(c, a;3)}(b), \ \overline R_{(a,b)}(c)=H^{-1}_{(c, a;2)}(b).$$
\end{definition}

As shown in Fig.~\ref{crossing1}, $\underline L$ (resp. $\overline L$, $\underline R$, $\overline R$) is related to  colors of the regions in the left side for the under-semi-arcs (resp. in the left side for the over-semi-arcs, in the right side for the  under-semi-arcs, in the right side for the over-semi-arcs) of a crossing.

\begin{figure}[h]
  \begin{center}
    \includegraphics[clip,width=7.0cm]{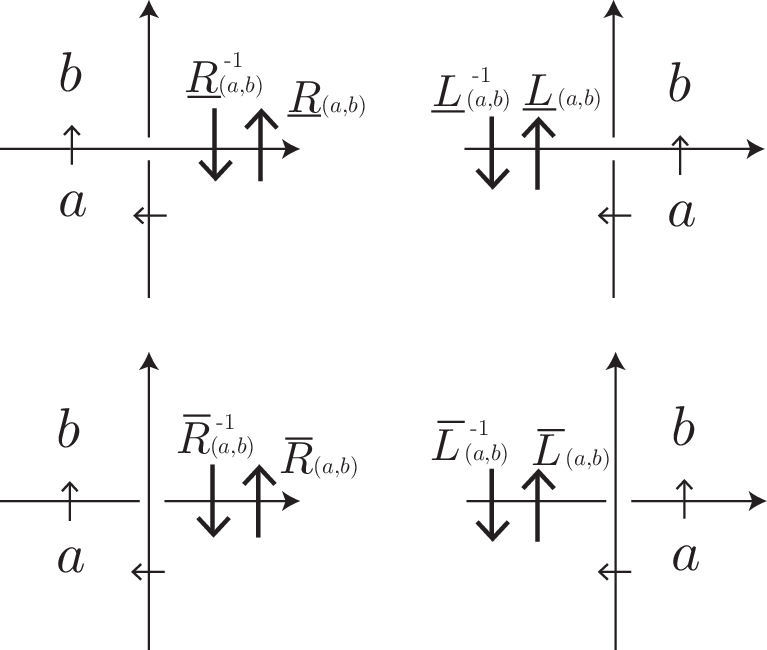}
    \caption{The maps $\underline{L}$, $\overline{L}$, $\underline{R}$, and $\overline{R}$}
    \label{crossing1}
  \end{center}
\end{figure}

\begin{lemma}\label{lem:LR}
For any element $a, b \in X$, $\underline L_{(a,b)}$, $\overline L_{(a,b)}$, $\underline R_{(a,b)}$, and $\overline R_{(a,b)}$ are bijective.
\end{lemma}
\begin{proof}
The inverses are defined by
$$\underline L^{-1}_{(a,b)}(c)=H^{-1}_{(a, b;2)}(c), \ \overline L^{-1}_{(a,b)}(c)=H^{-1}_{(a, b;3)}(c),$$
$$\underline R^{-1}_{(a,b)}(c)=H^{-1}_{(a, c;1)}(b), \ \overline R^{-1}_{(a,b)}(c)=H^{-1}_{(c,a;1)}(b).$$

\end{proof}

\begin{remark}\label{lem:crossing}
For any elements $a, b, c\in X$, we have $$\underline R_{(a,b)}(c)=\overline L^{-1}_{(c,a)}(b), \ \underline R^{-1}_{(a,b)}(c)= \overline R^{-1}_{(c,b)}(a)$$
$$\underline L_{(a,b)}(c)=\overline L_{(a,c)}(b), \ \underline L^{-1}_{(a,b)}(c)=\overline R_{(b,c)}(a), $$
see Fig.~\ref{crossing2}.
\begin{figure}[h]
  \begin{center}
    \includegraphics[clip,width=9.0cm]{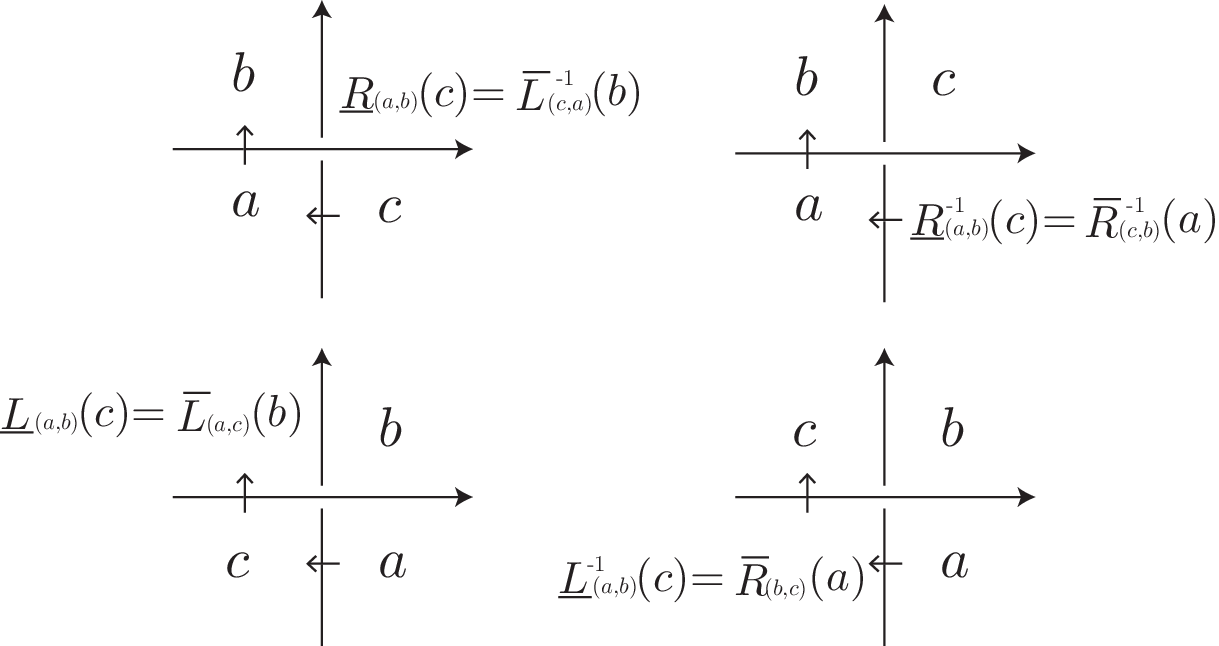}
    \caption{Properties of $\underline{L}$, $\overline{L}$, $\underline{R}$, and $\overline{R}$}
    \label{crossing2}
  \end{center}
\end{figure}
\end{remark}

Hereafter, $\underline L^{+1}$, $\overline L^{+1}$, $\underline R^{+1}$, and $\overline R^{+1}$ represent $\underline L$, $\overline L$, $\underline R$, and $\overline R$, respectively.
\begin{definition}\label{def:palettes1}
We put two copies $\{a^{+1} ~|~ a \in X\}$ and $\{a^{-1} ~|~ a \in X\}$ of $X$, and set $\mathcal{X}=\{a^{+1} ~|~ a \in X\} \cup \{a^{-1} ~|~ a \in X\}$, where we note that we distinguish two copies $a^{+1}$ and $a^{-1}$ of the same element $a$ of $X$ by putting the superscripts $+1$ and $-1$.
An {\it oriented $\mathcal{R}$-palette} $P$ of $(X, [\, ] )$ is a subset of $\displaystyle 
\bigcup_{n \in \mathbb{Z}_{+}} \mathcal{X}^n$ satisfying the following conditions:

(i) \ If $(a_1^{\varepsilon_1}, a_2^{\varepsilon_2}, \ldots, a_n^{\varepsilon_n}) \in P$, then
$(a_2^{\varepsilon_2}, \ldots, a_n^{\varepsilon_n}, a_1^{\varepsilon_1}) \in P$.

(ii) \ If $(a_1^{\varepsilon_1}, a_2^{\varepsilon_2}, \ldots, a_n^{\varepsilon_n}) \in P$, then it holds that

$$\underline L^{\varepsilon_n}_{(a_n,a_1; \varepsilon_n)} \circ \underline L^{\varepsilon_{n-1}}_{(a_{n-1},a_n; \varepsilon_{n-1})} \circ \cdots \circ \underline L^{\varepsilon_1}_{(a_1,a_2; \varepsilon_1)}={\rm id},$$

 $$\overline L^{\varepsilon_n}_{(a_n,a_1; \varepsilon_n)} \circ \overline L^{\varepsilon_{n-1}}_{(a_{n-1},a_n; \varepsilon_{n-1})} \circ \cdots \circ \overline L^{\varepsilon_1}_{(a_1,a_2; \varepsilon_1)}={\rm id},$$

 $$\underline R^{\varepsilon_n}_{(a_n,a_1; \varepsilon_n)} \circ \underline R^{\varepsilon_{n-1}}_{(a_{n-1},a_n; \varepsilon_{n-1})} \circ \cdots \circ \underline R^{\varepsilon_1}_{(a_1,a_2; \varepsilon_1)}={\rm id},$$

$$\overline R^{\varepsilon_n}_{(a_n,a_1; \varepsilon_n)} \circ \overline R^{\varepsilon_{n-1}}_{(a_{n-1},a_n; \varepsilon_{n-1})} \circ \cdots \circ \overline R^{\varepsilon_1}_{(a_1,a_2; \varepsilon_1)}={\rm id},$$

 where 
 $(a_i ,a_j; \varepsilon_i) = \left\{
\begin{array}{ll}
(a_i ,a_j)  & \mbox{ if } \varepsilon_i =+1, \\[5pt]
(a_j ,a_{i}) & \mbox{ if } \varepsilon_i =-1.\\[5pt]
\end{array}
\right.$

(iii) \ If $(a_1^{\varepsilon_1}, a_2^{\varepsilon_2}, \ldots, a_n^{\varepsilon_n}) \in P$, then it holds that
$$\Big(a, \underline L^{\varepsilon_1}_{(a_1,a_2; \varepsilon_1)}(a),  \underline L^{\varepsilon_2}_{(a_2,a_3; \varepsilon_2)} \circ \underline L^{\varepsilon_1}_{(a_1,a_2; \varepsilon_1)}(a), \ldots ,\underline L^{\varepsilon_{n-1}}_{(a_{n-1},a_n; \varepsilon_{n-1})} \circ \cdots \circ \underline L^{\varepsilon_1}_{(a_1,a_2; \varepsilon_1)}(a)\Big) \in P,$$
$$\Big(a, \overline L^{\varepsilon_1}_{(a_1,a_2; \varepsilon_1)}(a),  \overline L^{\varepsilon_2}_{(a_2,a_3; \varepsilon_2)} \circ \overline L^{\varepsilon_1}_{(a_1,a_2; \varepsilon_1)}(a), \ldots ,\overline L^{\varepsilon_{n-1}}_{(a_{n-1},a_n; \varepsilon_{n-1})} \circ \cdots \circ \overline L^{\varepsilon_1}_{(a_1,a_2; \varepsilon_1)}(a)\Big) \in P,$$
$$\Big(a, \underline R^{\varepsilon_1}_{(a_1,a_2; \varepsilon_1)}(a),  \underline R^{\varepsilon_2}_{(a_2,a_3; \varepsilon_2)} \circ \underline R^{\varepsilon_1}_{(a_1,a_2; \varepsilon_1)}(a), \ldots ,\underline R^{\varepsilon_{n-1}}_{(a_{n-1},a_n; \varepsilon_{n-1})} \circ \cdots \circ \underline R^{\varepsilon_1}_{(a_1,a_2; \varepsilon_1)}(a)\Big) \in P,$$
$$\Big(a, \overline R^{\varepsilon_1}_{(a_1,a_2; \varepsilon_1)}(a),  \overline R^{\varepsilon_2}_{(a_2,a_3; \varepsilon_2)} \circ \overline R^{\varepsilon_1}_{(a_1,a_2; \varepsilon_1)}(a), \ldots ,\overline R^{\varepsilon_{n-1}}_{(a_{n-1},a_n; \varepsilon_{n-1})} \circ \cdots \circ \overline R^{\varepsilon_1}_{(a_1,a_2; \varepsilon_1)}(a)\Big) \in P$$
for any $a \in X$. 

(iv) \ If $(a_1^{\varepsilon_1}, a_2^{\varepsilon_2},\ldots, a_n^{\varepsilon_n}) \in P$, when $n>2$, 
$$(a_1^{\varepsilon_2}, \big( \underline L^{\varepsilon_2}_{(a_2,a_3; \varepsilon_2)}(a_1) \big)^{\varepsilon_1}, a_3^{\varepsilon_3}, \ldots, a_n^{\varepsilon_n}) \in P,$$ 
$$(a_1^{\varepsilon_2}, \big( \overline L^{\varepsilon_2}_{(a_2,a_3; \varepsilon_2)}(a_1) \big)^{\varepsilon_1}, a_3^{\varepsilon_3}, \ldots, a_n^{\varepsilon_n}) \in P,$$ 
$$(a_1^{\varepsilon_2}, \big( \underline R^{\varepsilon_2}_{(a_2,a_3; \varepsilon_2)}(a_1) \big)^{\varepsilon_1}, a_3^{\varepsilon_3}, \ldots, a_n^{\varepsilon_n}) \in P,$$ 
$$(a_1^{\varepsilon_2}, \big( \overline R^{\varepsilon_2}_{(a_2,a_3; \varepsilon_2)}(a_1) \big)^{\varepsilon_1}, a_3^{\varepsilon_3}, \ldots, a_n^{\varepsilon_n}) \in P$$ hold, and when $n=2$, 
$$(a_1^{\varepsilon_2}, \big( \underline L^{\varepsilon_2}_{(a_2,a_1; \varepsilon_2)}(a_1) \big)^{\varepsilon_1}) \in P,$$ 
$$(a_1^{\varepsilon_2}, \big( \overline L^{\varepsilon_2}_{(a_2,a_1; \varepsilon_2)}(a_1) \big)^{\varepsilon_1}) \in P,$$ 
$$(a_1^{\varepsilon_2}, \big( \underline R^{\varepsilon_2}_{(a_2,a_1; \varepsilon_2)}(a_1) \big)^{\varepsilon_1}) \in P,$$ 
$$(a_1^{\varepsilon_2}, \big( \overline R^{\varepsilon_2}_{(a_2,a_1; \varepsilon_2)}(a_1) \big)^{\varepsilon_1}) \in P$$ 
hold.
\end{definition}

The axioms of an oriented $\mathcal{R}$-palette $P$ of $(X,[\,])$ are obtained from the oriented Reidemeister moves, of spatial graph diagrams. 
This is observed when we consider $(X,P)$-colorings for regions of spatial graph diagrams as follows:  
\begin{definition}\label{def:coloring1}
Let $(X,[\,])$ be a knot-theoretic ternary-quasigroup and $P$ be an oriented $\mathcal{R}$-palette of $(X,[\,])$.
Let $D$ be a diagram of an oriented spatial graph and $\mathcal{R}(D)$ the set of regions of $D$.
An {\it $(X,P)$-coloring} of $D$ is a map $C: \mathcal{R}(D) \to X$ satisfying the following conditions: 
\begin{itemize}
\item For a crossing $c$ with regions $r_1, r_2, r_3$, and $r_4$ as depicted in Fig.~\ref{coloring5},
\[
\Big[C(r_1), C(r_2), C(r_3)\Big] = C(r_4)
\]
holds, see also Fig.~\ref{coloring2}, where $r_1$ is located in the right side of the over- and under-arcs both with respect to their orientations, $r_2$ is the region which is adjacent to $r_1$ by an under-arc and $r_3$ is the region which is adjacent to $r_1$ by the over-arc.
\item For a vertex $v$ with regions $r_1, \ldots , r_n$ that appear clockwisely  as depicted in Fig.~\ref{coloring5}, 
\[
\Big(C(r_1)^{\varepsilon_1}, C(r_2)^{\varepsilon_2}, \ldots , C(r_n)^{\varepsilon_n}\Big) \in P
\]
holds, where for $i\in \{1,\ldots ,n\}$, ${\varepsilon_i}=+1$ if the arc $x_i$ which separates the regions $r_{i}$ and $r_{i+1}$ points to $v$, and $\varepsilon_i=-1$ otherwise. 
\end{itemize}
We call $C(r)$ the {\it color} of a region $r$.
We denote by ${\rm Col}_{(X,P)}(D)$ the set of $(X,P)$-colorings of $D$.
\end{definition}
\begin{figure}[h]
  \begin{center}
    \includegraphics[clip,width=9.0cm]{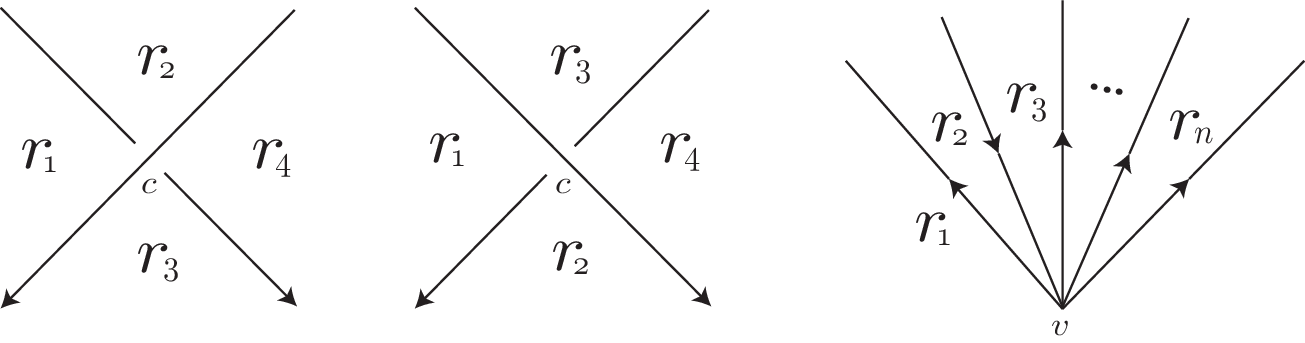}
    \caption{Crossings and vertices}
    \label{coloring5}
  \end{center}
\end{figure}
\begin{proposition}\label{prop:coloring1}
Let $D$ and $D'$ be diagrams of oriented spatial graphs. If $D$ and $D'$ represent the same spatial graph, then there exists a bijection between ${\rm Col}_{(X,P)}(D)$ and ${\rm Col}_{(X,P)}(D')$.  
\end{proposition}
\begin{proof}
Let $D$ and $D'$ be diagrams such that $D'$ is obtained from $D$ by a single Reidemeister move. 
Let $E$ be a $2$-disk in which the move is applied. 
Let $C$ be an $(X,P)$-coloring of $D$. 
We define an $(X,P)$-coloring $C'$ of $D'$, corresponding to $C$, by $C'(r) = C(r)$ for a
region $r$ appearing in the outside of $E$. Then the colors of the regions appearing in $E$,
by $C'$, are uniquely determined, see Fig.~\ref{R4} and \ref{R5}. For example, in the upper move in Fig.~\ref{R4}, since 
\begin{align*}
&\Big(x, \overline R^{-1}_{(a_2,a_1)}(x), \ldots ,\overline R_{(a_{n-1},a_n)} \circ \cdots \circ \overline R_{(a_2,a_3)} \circ \overline R^{-1}_{(a_2,a_1)}(x)\Big)\\
&=\Big(x, \overline R^{-1}_{(a_1,a_2; -1)}(x), \ldots ,\overline R^{+1}_{(a_{n-1},a_n; +1)} \circ \cdots \circ \overline R^{+1}_{(a_2,a_3; +1)} \circ \overline R^{-1}_{(a_1,a_2; -1)}(x)\Big) \in P
\end{align*}
for any $(a_1^{\varepsilon_1}, a_2^{\varepsilon_2}, \ldots, a_n^{\varepsilon_n}) \in P$ and $x \in X$, by (iii) of Definition ~\ref{def:palettes1}, if the vertex in the left picture satisfies the coloring condition, so does the vertex in the right picture.

 Thus we have a bijection ${\rm Col}_{(X, P)}(D) \to {\rm Col}_{(X, P)}(D') $ that maps $C$ to $C'$.
\end{proof}

\begin{figure}[h]
  \begin{center}
    \includegraphics[clip,width=12.0cm]{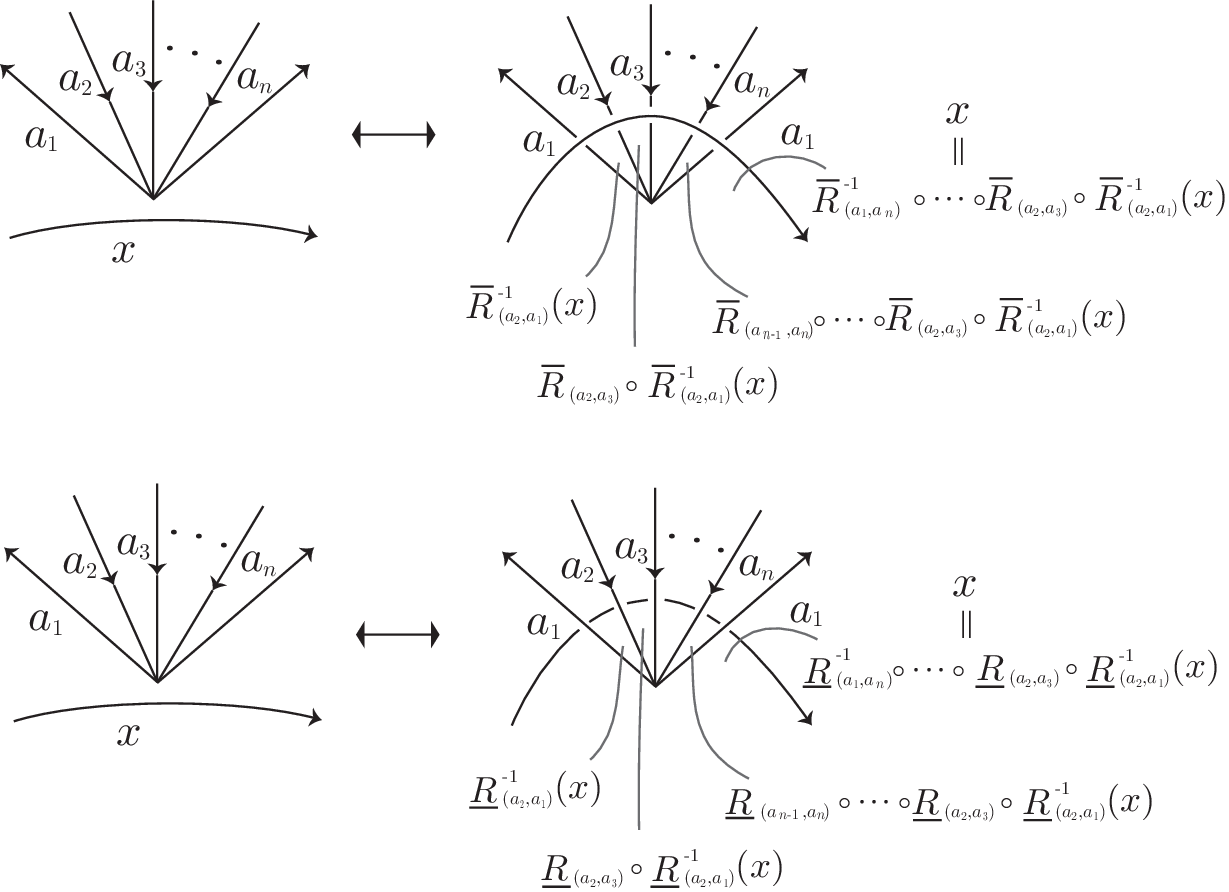}
    \caption{The correspondence of $(X, P)$-colorings under $R$IV}
    \label{R4}
  \end{center}
\end{figure}
\begin{figure}[h]
  \begin{center}
    \includegraphics[clip,width=8.0cm]{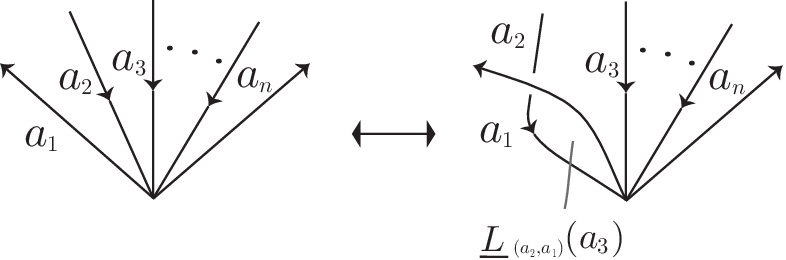}
    \caption{The correspondence of $(X, P)$-colorings under $R$V}
    \label{R5}
  \end{center}
\end{figure}

\section{Generalized $\mathcal{R}$-palettes for region colorings of unoriented spatial graph diagrams}\label{generalunoriented}

The notion of $\mathcal{R}$-palettes for Dehn $p$-colorings can be extended for knot-theoretic ternary-quasigroups and  region colorings of ``unoriented" spatial graph diagrams in general, which can be also regarded as a simplification of the notion shown in the previous section. 
In this section, we will show this generalization.

\begin{definition}
Let $(X, [\,])$ be a knot-theoretic ternary-quasigroup.
The ternary operation $[\,]$ is said to be {\it unoriented} if it satisfies that 
\begin{enumerate}
\item[] \hspace{-0.5cm}($\mathcal{KTQ}$3) For any $a,b,c\in X$, 
\[
\begin{array}{l}
[b,a,[a,b,c]]=c \mbox{ \ \ and \ \  } [c,[a,b,c], a] =b.
\end{array} 
\]
\end{enumerate} 
\end{definition}
\noindent A knot-theoretic ternary-quasigroup  $(X, [\,])$ with an unoriented $[\,]$ gives a region coloring of an unoriented link diagram by $(X, [\,])$, that is, an assignment of an element of $X$ to each region satisfying the crossing condition depicted in Fig.~\ref{coloring1}. 
Note that this condition does not depend on the choice of the specified region labeled by $\star$, see Fig.~\ref{coloring6}.
\begin{figure}[h]
  \begin{center}
    \includegraphics[clip,width=7.0cm]{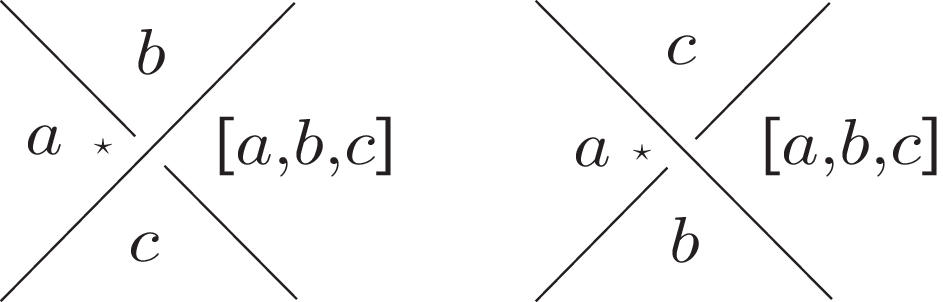}
    \caption{The crossing condition for a region coloring by $(X,[\,])$}
    \label{coloring1}
  \end{center}
\end{figure}
\begin{figure}[h]
  \begin{center}
    \includegraphics[clip,width=10.0cm]{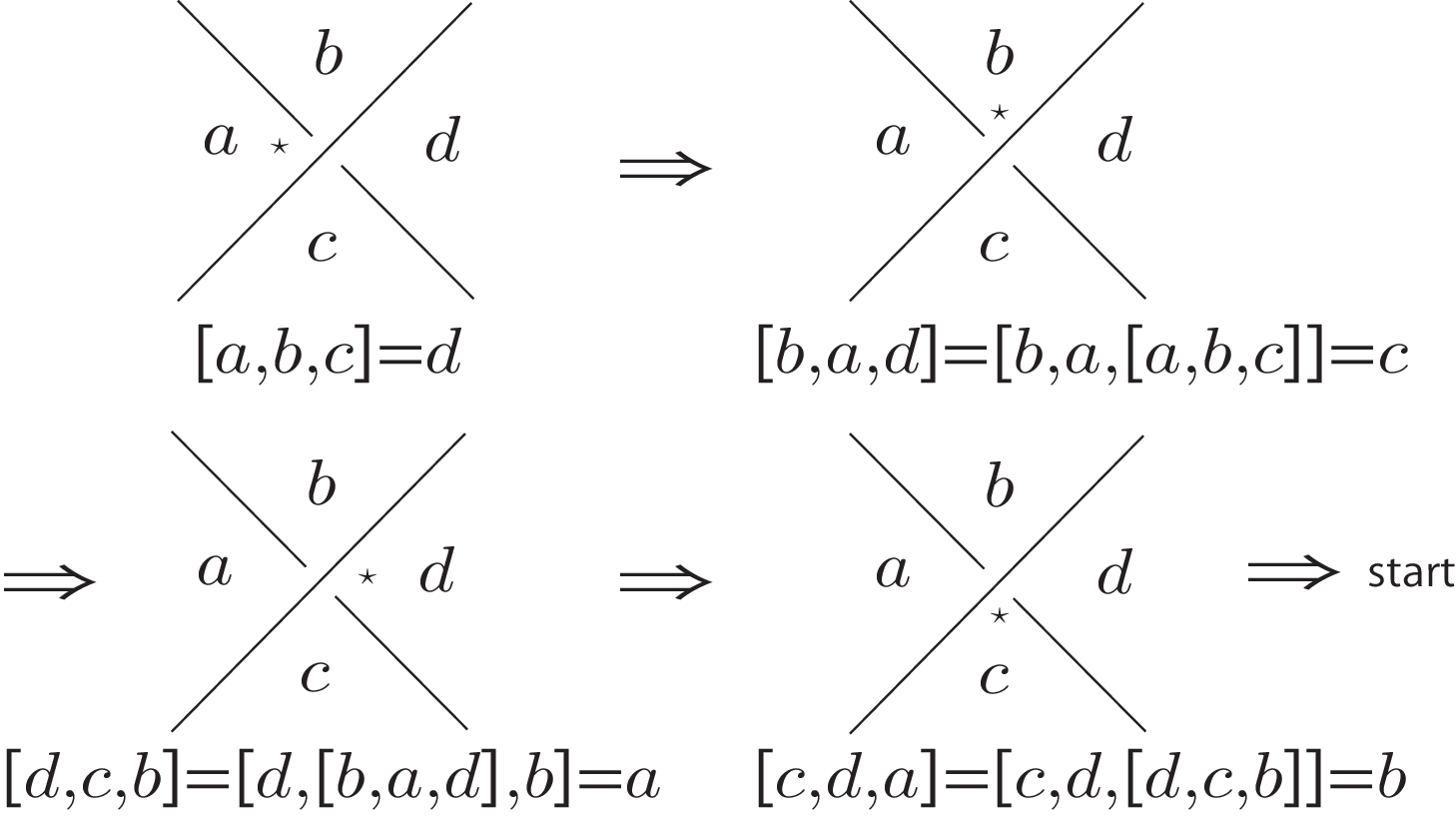}
    \caption{The crossing condition is independent of the choice of the specified region}
    \label{coloring6}
  \end{center}
\end{figure}

Let $(X, [\,])$ be a knot-theoretic ternary-quasigroup with an unoriented $[\,]$.
Note again that $H_{(a,b;i)}:X \to X$ with $i \in \{2,3\}$ is defined by $$H_{(a, b;2)}(c)=[a,c,b] \mbox{ and } \ H_{(a, b;3)}(c)=[a,b,c]$$ for $a,b,c\in X$. 
\begin{definition}\label{def:palettes2}
An {\it unoriented $\mathcal{R}$-palette} $P$ of $(X, [\, ] )$  is a subset of $\displaystyle 
\bigcup_{n \in \mathbb{Z}_{+}} {X}^n$ satisfying the following conditions:

(i) \ If $(a_1, a_2, \ldots, a_n) \in P$, then
$(a_2, \ldots, a_n, a_1) \in P$.

(ii) \ If $(a_1, a_2, \ldots, a_n) \in P$, then it holds that

$$H_{(a_n,a_1; 2)} \circ H_{(a_{n-1},a_n; 2)} \circ \cdots \circ H_{(a_1,a_2; 2)}={\rm id},$$

$$H_{(a_n,a_1; 3)} \circ H_{(a_{n-1},a_n; 3)} \circ \cdots \circ H_{(a_1,a_2; 3)}={\rm id}.$$

(iii) \ If $(a_1, a_2, \ldots, a_n) \in P$, then it holds that
$$\Big(a, H_{(a_1,a_2; 2)}(a),  H_{(a_2,a_3; 2)} \circ  H_{(a_1,a_2; 2)}(a), \ldots , H_{(a_{n-1},a_n; 2)} \circ \cdots \circ  H_{(a_1,a_2; 2)}(a)\Big) \in P,$$
$$\Big(a, H_{(a_1,a_2; 3)}(a),  H_{(a_2,a_3; 3)} \circ  H_{(a_1,a_2; 3)}(a), \ldots ,  H_{(a_{n-1},a_n; 3)} \circ \cdots \circ  H_{(a_1,a_2; 3)}(a)\Big) \in P$$
for any $a \in X$. 

(iv) \ If $(a_1, a_2, \ldots, a_n) \in P$, when $n>2$, 
$$(a_1, [a_2,a_3,a_1], a_3, \ldots, a_n) \in P,$$ 
$$(a_1, [a_2,a_1,a_3], a_3, \ldots, a_n) \in P$$ 
hold, and when $n=2$, 
$$(a_1, [a_2,a_1,a_1]) \in P$$
holds.
\end{definition}

The axioms of an unoriented $\mathcal{R}$-palette $P$ of $(X,[\,])$ are obtained from the unoriented Reidemeister moves of spatial graph diagrams, which are observed when we consider $(X,P)$-colorings for regions of spatial graph diagrams as follows:  
\begin{definition}\label{def:coloring1}
Let $(X,[\,])$ be a knot-theoretic ternary-quasigroup with an unoriented $[\,]$ and $P$ be an unoriented $\mathcal{R}$-palette of $(X,[\,])$.
Let $D$ be a diagram of an unoriented spatial graph and $\mathcal{R}(D)$ the set of regions of $D$.
An {\it $(X,P)$-coloring} of $D$ is a map $C: \mathcal{R}(D) \to X$ satisfying the following conditions: 
\begin{itemize}
\item For a crossing $c$ with regions $r_1, r_2, r_3$, and $r_4$ as depicted in Fig.~\ref{coloring7},
\[
\Big[C(r_1), C(r_2), C(r_3)\Big] = C(r_4)
\]
holds, where for an arbitrary chosen region $r_1$, $r_2$ is the region which is adjacent to $r_1$ by an under-arc and $r_3$ is the region which is adjacent to $r_1$ by the over-arc.
 See also Fig.~\ref{coloring1}.
\item For a vertex $v$ with regions $r_1, \ldots , r_n$ that appear clockwisely as depicted in Fig.~\ref{coloring7}, 
\[
\Big(C(r_1), C(r_2), \ldots , C(r_n)\Big) \in P
\]
holds. 
\end{itemize}
We call $C(r)$ the {\it color} of a region $r$.
We denote by ${\rm Col}_{(X,P)}(D)$ the set of $(X,P)$-colorings of $D$.
\end{definition}

\begin{proposition}
Let $D$ and $D'$ be diagrams of unoriented spatial graphs. If $D$ and $D'$ represent the same spatial graph, then there exists a bijection between ${\rm Col}_{(X,P)}(D)$ and ${\rm Col}_{(X,P)}(D')$.  
\end{proposition}
\begin{proof}
This is proven similarly as the proof of Proposition~\ref{prop:coloring1}.
\end{proof}

\begin{remark}
In this paper, we study $\mathcal{R}$-palettes for Dehn $p$-colorings, which coincide with $\mathcal{R}$-palettes for
the ternary-quasigroup $(\mathbb Z_p, [\,])$ with the operation $[\,]$ defined by $[a,b,c]=a-b+c$. 
Since $[b,a,[a,b,c]]=b-a+(a-b+c) = c$ and $[c,[a,b,c], a] =c-(a-b+c)+a=b$ for $a,b,c \in \mathbb Z_p$, the ternary operation $[\,]$ is unoriented. 

For this case, 
\begin{align*}
&\mbox{the axiom (ii) of Definition~\ref{def:palettes2} holds}\\ 
&\iff H_{(a_n,a_1; 2)} \circ H_{(a_{n-1},a_n; 2)} \circ \cdots \circ H_{(a_1,a_2; 2)}(a)=a \mbox{ for all $a\in \mathbb Z_p$}
\\
&\iff \left\{
\begin{array}{lll}
2a_1-a=a & \mbox{ for all $a\in \mathbb Z_p$ if $n$ is odd}
,\\
a=a & \mbox{ for all  $a\in \mathbb Z_p$ if $n$ is even.} 
\end{array}
\right. 
\end{align*}
Since there does not exist $(a_1, \ldots , a_n)\in \mathbb Z_p^n$ such that $2a_1-a=a$ for any $a\in \mathbb Z_p$, the axiom (ii) of Definition~\ref{def:palettes2} implies that any palette of the ternary-quasigroup $(\mathbb Z_p, [\,])$ is a subset of $\bigcup_{n\in 2\mathbb{Z}_+} \mathbb{Z}_p^n$.

For the axiom (iii) of Definition~\ref{def:palettes2},
$$\Big(a, H_{(a_1,a_2; 2)}(a),  H_{(a_2,a_3; 2)} \circ  H_{(a_1,a_2; 2)}(a), \ldots , H_{(a_{n-1},a_n; 2)} \circ \cdots \circ  H_{(a_1,a_2; 2)}(a)\Big) \in P$$
means 
$$\Big(a, a_2 + (-1)^2(a_1-a), \ldots ,a_i + (-1)^i(a_1-a), \ldots , a_n + (-1)^n(a_1-a) \Big) \in P,$$
and 
$$\Big(a, H_{(a_1,a_2; 3)}(a),  H_{(a_2,a_3; 3)} \circ  H_{(a_1,a_2; 3)}(a), \ldots ,  H_{(a_{n-1},a_n; 3)} \circ \cdots \circ  H_{(a_1,a_2; 3)}(a)\Big) \in P$$
means 
$$\Big(a, a_1- a_2 +a, \ldots ,a_1- a_i +a, \ldots , a_1- a_n +a \Big) \in P.$$

For the axiom (iv) of Definition~\ref{def:palettes2}, when $n>2$, 
$$(a_1, [a_2,a_1,a_3], a_3, \ldots, a_n) \in P$$ 
means 
$$
(a_1, -a_1 + a_2 + a_3, a_3, \ldots, a_n) \in P.
$$
By repeating this operation
$(a_1, a_2, \ldots , a_n) \in P \Longrightarrow (a_1, [a_2,a_1,a_3], a_3, \ldots, a_n) \in P$
$p-1$ times,
$$(a_1, a_2, \ldots , a_n) \in P \Longrightarrow (a_1, [a_2,a_3,a_1], a_3, \ldots, a_n) \in P$$ 
is obtained. Note that 
$(a_1, [a_2,a_1,a_1]) \in P$ means $(a_1, a_2) \in P$ when $n=2$.

By the above observation, the definition of an unoriented $\mathcal{R}$-palette for the ternary-quasigroup $(\mathbb Z_p, [\,])$ can be rewritten as Definition~\ref{def:DihedralRpalettes}, which is the simplified definition of an $\mathcal{R}$-palette for the ternary-quasigroup $(\mathbb Z_p, [\,])$ with $[a,b,c]=a-b+c$.  
\end{remark}

\end{document}